\documentclass{amsart}
\usepackage{amssymb, enumerate}
\usepackage[all]{xy}

\usepackage{hyperref}

\def\RCS$#1: #2 ${\expandafter\def\csname RCS#1\endcsname{#2}}
\RCS$Revision: 233 $
\RCS$Date: 2010-05-24 20:23:27 +0100 (Mon, 24 May 2010) $

\newcommand{\Cla}{C^{\rm la}}
\newcommand{\Dla}{D^{\rm la}}
\newcommand{\Can}{C^{\rm an}}
\newcommand{\HH}{\mathcal{H}}
\newcommand{\into}{\hookrightarrow}
\newcommand{\hotimes}[1]{\, \mathop{\widehat\otimes}_{#1}\, }
\newcommand{\p}{\mathfrak{p}}

\DeclareMathOperator{\End}{End}
\DeclareMathOperator{\Ord}{Ord}
\DeclareMathOperator{\Exp}{Exp}
\DeclareMathOperator{\Ind}{Ind}
\DeclareMathOperator{\Rep}{Rep}
\DeclareMathOperator{\Hom}{Hom}
\DeclareMathOperator{\Lie}{Lie}
\DeclareMathOperator{\Spec}{Spec}

\newtheorem{theorem}{Theorem}[section]

\newtheorem{lemma}[theorem]{Lemma}
\newtheorem{definition}[theorem]{Definition}
\newtheorem{proposition}[theorem]{Proposition}
\newtheorem{corollary}[theorem]{Corollary}

\newtheorem*{theorem*}{Theorem}
\newtheorem*{corollary*}{Corollary}

\theoremstyle{remark}
\newtheorem*{remark}{Remark}
\newtheorem*{remarks}{Remarks}

\numberwithin{equation}{section}

\title[Emerton's Jacquet functors]{Emerton's Jacquet functors for non-Borel parabolic subgroups}
\author{Richard Hill}
\address{Department of Mathematics \\ University College London \\ Gower Street \\ London WC1E 6BT, UK}
\email{rih@math.ucl.ac.uk}

\author{David Loeffler}
\address{Warwick Mathematics Institute \\ Zeeman Building \\ University of Warwick \\ Coventry CV4 7AL, UK}
\email{d.loeffler.01@cantab.net}

\thanks{The second author is grateful for the support of EPSRC Postdoctoral Fellowship EP/F04304X/2.}

\begin{document}

\begin{abstract}
This paper studies Emerton's Jacquet module functor for locally analytic representations of $p$-adic reductive groups, introduced in \cite{emerton-jacquet}. When $P$ is a parabolic subgroup whose Levi factor $M$ is not commutative, we show that passing to an isotypical subspace for the derived subgroup of $M$ gives rise to essentially admissible locally analytic representations of the torus $Z(M)$, which have a natural interpretation in terms of rigid geometry. We use this to extend the construction in of eigenvarieties in \cite{emerton-interpolation} by constructing eigenvarieties interpolating automorphic representations whose local components at $p$ are not necessarily principal series.
\end{abstract}

\maketitle

\section{Introduction}

\subsection{Background}

Let $\mathfrak{G}$ be a reductive group over a number field $F$. The automorphic representations of the group $\mathfrak{G}(\mathbb{A})$, where $\mathbb{A}$ is the ad\`ele ring of $F$, are central objects of study in number theory. In many cases, it is expected that the set $\Pi(\mathfrak{G})$ of automorphic representations contains a distinguished subset $\Pi(\mathfrak{G})^{\rm arith}$ of representations which are (in some sense) ``definable over $\overline{\mathbb{Q}}$''. The subject of this paper is the $p$-adic interpolation properties of these representations (and their associated Hecke eigenvalues). Following the pioneering work of Coleman and Coleman-Mazur \cite{C-CO,Cpadic,CMeigen} for the automorphic representations attached to modular forms with nonzero Hecke eigenvalue at $p$, it is expected that these Hecke eigenvalues should be parametrised by $p$-adic rigid spaces (eigenvarieties).

A very general construction of eigenvarieties is provided by the work of Emerton \cite{emerton-interpolation}, using the cohomology of arithmetic quotients of $\mathfrak{G}$. For any fixed open compact subgroup $K_f \subseteq \mathfrak{G}(\mathbb{A}_f)$ (where $\mathbb{A}_f$ is the finite ad\`eles of $F$), and $K_\infty^\circ$ the identity component of a maximal compact subgroup of $\mathfrak{G}(F \otimes \mathbb{R})$, the quotients $Y(K_f) = \mathfrak{G}(F) \backslash \mathfrak{G}(\mathbb{A}) / K_f K_\infty^\circ$ are real manifolds, equipped with natural local systems $\mathcal{V}_X$ for each algebraic representation $X$ of $\mathfrak{G}$. The cohomology groups $H^i(Y(K_f), \mathcal{V}_X)$ are finite-dimensional, and passing to the direct limit over $K_f$ gives an admissible smooth representation $H^i(\mathcal{V}_X)$ of $\mathfrak{G}(\mathbb{A}_f)$. Every irreducible subquotient of $H^i(\mathcal{V}_X)$ is the finite part of an automorphic representation (whose infinite part is determined by $X$); we say that the representations arising in this way are {\it cohomological} (in degree $i$).

Emerton's construction proceeds in two major steps. Fix a prime $\p$ above $p$ and an open compact subgroup $K^{(\p)} \subseteq \mathfrak{G}(\mathbb{A}_f^{(\p)})$ (a ``tame level''). Firstly, from the spaces $H^i(Y(K^{(\p)} K_{\p}), \mathcal{V}_X)$ for various open compact subgroups $K_\p \subseteq G = \mathfrak{G}(F_\p)$, Emerton constructs Banach space representations $\tilde H^i(K^{(\p)})$ of $G$. For any complete subfield $L$ of $F_\p$, the spaces $\tilde H^i(K^{(\p)})_{\rm la}$ of locally $L$-analytic vectors are locally $L$-analytic representations of $G$, and there are natural maps
\begin{equation}\label{eq:edgemap}
 H^i(\mathcal{V}_X)^{K^{(\p)}} \to \Hom_{\mathfrak{g}}(X', \tilde H^i(K^{(\p)})_{\rm la})
\end{equation}
where $\mathfrak{g} = \operatorname{Lie} G$. In many cases, these maps are known to be isomorphisms; if this holds, the automorphic representations which are cohomological in degree $i$ are exactly those which appear as subquotients of $\Hom_{\mathfrak{g}}(X', \tilde H^i(K^{(\p)})_{\rm la})$ for some $X$ and tame level $K^{(\p)}$.

The second step in the construction is to extract the desired information from the space $\tilde H^i(K^{(\p)})_{\rm la}$. This is carried out by applying the Jacquet module functor of \cite{emerton-jacquet}, for a Borel subgroup $B \subseteq G$. This then produces an essentially admissible locally analytic representation of the Levi factor $M$ of $B$, which is a torus. There is an anti-equivalence of categories between essentially admissible locally analytic representations of $M$ and coherent sheaves on the rigid-analytic space $\widehat M$ parametrising characters of $M$. The eigenvariety $E(i, K^{(\p)})$ is then constructed from this sheaf by passing to the relative spectrum of the unramified Hecke algebra $\HH^{\rm sph}$ of $K^{(\p)}$; points of this variety correspond to characters $(\kappa, \lambda) \in \widehat M \times \Spec \HH^{\rm sph}$ such that the $(M = \kappa, \HH^{\rm sph} = \lambda)$-eigenspace of $J_B(\tilde H^i(K^{(\p)})_{\rm la})$ is nonzero. Hence if the map \eqref{eq:edgemap} above is an isomorphism, there is a point of $E(i, K^{(\p)})$ for each automorphic representation $\pi = \bigotimes_v \pi_v$ which is cohomological in degree $i$ with $(\pi_f^{(\p)})^{K^{(\p)}} \otimes J_B(\pi_\p) \ne 0$. 

\subsection{Statement of the main result}

In this paper, we consider the situation where $B$ is replaced by a general parabolic subgroup $P$ of $G$. This extends the scope of the theory in two ways: firstly, it may happen that no Borel subgroup exists ($G$ may not be quasi-split); and even if a Borel subgroup exists, there will usually be automorphic representations for which $J_B(\pi_\p) = 0$, which do not appear in Emerton's eigenvariety. 

As above, we choose a number field $F$, a connected reductive group $\mathfrak{G}$ over $F$, and a prime $\p$ of $F$ above the rational prime $p$. Let $\mathcal{G} = \mathfrak{G} \times_F F_\p$, a reductive group over $F_\mathfrak{p}$, and $G = \mathcal{G}(F_\p)$. Let us choose a parabolic subgroup $\mathcal{P}$ of $\mathcal{G}$ (not necessarily arising from a parabolic subgroup of $\mathfrak{G}$), with unipotent radical $\mathcal{N}$; and let $\mathcal{M}$ be a Levi factor of $\mathcal{P}$, with centre $\mathcal{Z}$ and derived subgroup $\mathcal{D}$. We write $G = \mathcal{G}(F_\p)$, and similarly for $P, M, D, Z$. We choose a complete extension $L$ of $\mathbb{Q}_p$ contained in $F_\p$, so $G, P, M, D, Z$ are locally $L$-analytic groups. 

Let $\Gamma = D \times \mathfrak{G}(\mathbb{A}_f^\p) \times \pi_0$, where $\pi_0$ is the component group of $\mathfrak{G}(F \otimes \mathbb{R})$. Let us choose an open compact subgroup $U \subseteq \Gamma$ (this is the most natural notion of a ``tame level'' in this context), and a finite-dimensional irreducible algebraic representation $W$ of $\mathcal{M}$. As we will explain below, the Hecke algebra $\HH(\Gamma // U)$ can be written as a tensor product $\HH^{\rm ram} \otimes \HH^{\rm sph}$, where $\HH^{\rm sph}$ is commutative, and $\HH^{\rm ram}$ is finitely-generated (and supported at a finite set of places $S$).

\begin{theorem*}[Theorem \ref{thm:maintheorem}] There exists a rigid-analytic subvariety $\mathcal{E}(i, P, W, U)$ of $\widehat Z \times \Spec \HH^{\rm sph}$, endowed with a coherent sheaf $\overline{\mathcal{F}}(i, P, W, U)$ with a right action of $\HH^{\rm ram}$, such that:
\begin{enumerate}
\item The natural projection $\mathcal{E}(i, P, W, U) \to \mathfrak{z}'$ has discrete fibres. In particular, the dimension of $\mathcal{E}(i, P, W, U)$ is at most equal to the dimension of $Z$.
\item The point $(\chi, \lambda) \in \widehat Z \times \Spec \HH^{\rm sph}$ lies in $\mathcal{E}(i, P, W, U)$ if and only if the $(Z = \chi, \HH^{\rm sph} = \lambda)$-eigenspace of $\Hom_U\left(W, J_P(\tilde H^i)_{\rm la}\right)$ is nonzero. If this is so, the fibre of $\overline{\mathcal{F}}(i, P, W, U)$ at $(\chi, \lambda)$ is isomorphic as a right $\HH^{\rm ram}$-module to the dual of that eigenspace.
\item If there is a compact open subgroup $G_0 \subseteq G$ such that $(\tilde H^i_{\rm la})^{U^{(\p)}}$ is isomorphic as a $G_0$-representation to a finite direct sum of copies of $\Cla(G_0)$ (where $U^{(\p)} = U \cap \mathfrak{G}(\mathbb{A}_f^\p)$), then $\mathcal{E}(i, P, W, U)$ is equidimensional, of dimension equal to the rank of $Z$.
\end{enumerate}
\end{theorem*}

Now let us suppose that $W$ is absolutely irreducible, and write $\Pi(i, P, W, U)$ for the set of irreducible smooth $\mathfrak{G}(\mathbb{A}_f) \times \pi_0$-representations $\pi_f$ such that $J_P(\pi_f)^U \ne 0$, and $\pi_f$ appears as a subquotient of the cohomology space $H^i(\mathcal{V}_{X})$ for some irreducible algebraic representation $X$ of $G$ such that $(X')^N \cong W \otimes \chi$ for a character $\chi$. To any such $\pi_f$, we may associate the point $(\theta \chi, \lambda) \in \widehat Z \times \Spec \HH^{\rm sph}$, where $\theta$ is the smooth character by which $Z$ acts on $J_P(\pi_p)$, and $\lambda$ the character by which $\HH^{\rm sph}$ acts on $J_P(\pi_f)^U$. Let $E(i, P, W, U)_{\rm cl}$ denote the set of points of $\widehat Z \times \Spec \HH^{\rm sph}$ obtained in this way from representations $\pi_f \in \Pi(i, P, W, U)$.

\begin{corollary*}[Corollary \ref{corr:maincorr}]
 If the map \eqref{eq:edgemap} is an isomorphism in degree $i$ for all irreducible algebraic representations $X$ such that $(X')^N$ is a twist of $W$, then $E(i, P, W, U)_{\rm cl} \subseteq \mathcal{E}(i, P, W, U)$. In particular, the Zariski closure of $E(i, P, W, U)_{\rm cl}$ has dimension at most $\dim Z$.
\end{corollary*}

In the special case when $\mathfrak{G}(F \otimes \mathbb{R})$ is compact modulo centre, a related statement has been proved (by very different methods) by the second author \cite{DL-algebraic}. 

If $P_1$ and $P_2$ are two different choices of parabolic, with $P_1 \supseteq P_2$, we have a relation between the eigenvarieties attached to $P_1$ and $P_2$ under a mild additional hypothesis, namely that the tame level be of the form $U^{(\p)} \times U_{\p}$, with $U^{(\p)}$ an open compact subgroup away from $\p$ and $U_\p$ an open compact subgroup of $D_1 = [M_1, M_1]$ which admits a certain decomposition with respect to the parabolic $P_2 \cap D_1$ (see \S \ref{ssect:hecke} below). In this situation, we have the following:

\begin{theorem*}[Theorem \ref{thm:comparison}]
If $U$ is of the above type, then the space $\mathcal{E}(i, P_1, W, U)$ is equal to the union of two subvarieties $\mathcal{E}(i, P_1, W, U)_{P_2-{\rm fs}}$ and $\mathcal{E}(i, P_1, W, U)_{P_2-{\rm null}}$, which are respectively endowed with sheaves of $\HH^{\rm ram}$-modules $\overline{\mathcal{F}}(i, P, W, U)_{P_2-{\rm fs}}$ and $\overline{\mathcal{F}}(i, P, W, U)_{P_2-{\rm null}}$ whose direct sum is $\overline{\mathcal{F}}(i, P, W, U)$. 

If $\pi_f \in \Pi(i, P, W, U)$ and $\pi_f$ is not annihilated by the map \eqref{eq:edgemap}, then the point of $\mathcal{E}(i, P_1, W, U)$ corresponding to $\pi_f$ lies in the former subvariety if $J_{P_2}(\pi_\p) \ne 0$, and in the latter if $J_{P_2}(\pi_\p) = 0$. Moreover, there is a closed subvariety of $\mathcal{E}(i, P_2, W^{N_{12}}, U \cap D_2)$ whose image in $\widehat Z_1 \times \Spec \HH^{\rm sph}$ is $\mathcal{E}(i, P_1, W, U)_{P_2-{\rm fs}}$.
\end{theorem*}

\section{Preliminaries}

\subsection{Notation and definitions}

Let $p$ be a prime. Let $K \supseteq \mathbb{Q}_p$ be a complete discretely valued field, which will be the coefficient field for all the representations we consider, and $L$ a finite extension of $\mathbb{Q}_p$ contained in $K$. If $V$ is a locally convex $K$-vector space, we let $V'$ denote the continuous dual of $V$. We write $V'_b$ for $V'$ endowed with the strong topology (which is the only topology on $V'$ we shall consider). 

Let $S$ be an abstract semigroup. A {\it topological representation} of $S$ is a locally convex Hausdorff topological $K$-vector space $V$ endowed with a left action of $S$ by continuous operators. If $S$ has a topology, we say that the representation is {\it separately continuous} if the orbit map of each $v \in V$ is a continuous map $S \to V$, and {\it continuous} if the map $S \times V \to V$ is continuous. In particular, this applies when $S$ is a topological $K$-algebra and $V$ is an $S$-module, in which case we shall refer to $V$ as a separately continuous or or continuous topological $S$-module.

If $G$ is a locally compact topological group and $V$ is a continuous representation of $G$, then $V'$ is a module over the algebra $D(G)$ of measures on $G$ \cite[5.1.7]{emerton-memoir}, defined as $C(G)'$ where $C(G)$ is the space of continuous $K$-valued functions on $G$. If $G$ is a locally $p$-adic analytic group, then for any open compact subgroup $H \subseteq G$, the subalgebra $D(H)$ is Noetherian, and we say $V$ is {\it admissible continuous} {\cite[Lemma 3.4]{ST-banach}} if $V$ is a Banach space and $V'$ is finitely generated over $D(H)$ for one (and hence every) open compact $H$.

If $G$ is a locally $L$-analytic group, in the sense of \cite{ST-distributions}, then we say the representation $V$ is {\it locally analytic} if it is a continuous $G$-representation on a space of compact type, and the orbit maps are locally $L$-analytic functions $G \to V$. This implies \cite[5.1.9]{emerton-memoir} that $V'_b$ is a separately continuous topological module over the topological $K$-algebra $\Dla(G)$ of distributions on $G$, defined as $\Cla(G)'_b$ where $\Cla(G)$ is the space of locally $L$-analytic $K$-valued functions on $G$. For $H$ an open compact subgroup, the subalgebra $\Dla(H)$ is a Fr\'echet-Stein algebra \cite[5.1]{ST-admissible}, so the category of coadmissible $\Dla(H)$-modules is defined \cite[\S 3]{ST-admissible}; we say $V$ is {\it admissible locally analytic} if $V'_b$ is coadmissible as a module over $\Dla(H)$ for one (and hence every) open compact $H$.

Finally, if $G$ is a locally $L$-analytic group for which $Z = Z(G)$ is topologically finitely generated, we say the representation $V$ is {\it $Z$-tempered} if it is locally analytic and can be written as an increasing union of $Z$-invariant $BH$-spaces. This implies that for any open compact subgroup $H \subseteq G$, $V'_b$ is a jointly continuous topological module over the algebra $D^{\rm ess}(H, Z(G)) = \Dla(H) \hotimes{\Dla(Z \cap H)} \Can(\widehat Z)$, where $\widehat Z$ is the rigid space\footnote{The space $\widehat Z$ is in fact defined over $L$, but we shall always consider it as a rigid space over $K$ by base extension.} parametrising characters of $Z$. The algebra $D^{\rm ess}(H, Z(G))$ is also a Fr\'echet-Stein algebra \cite[5.3.22]{emerton-memoir}, and we say $V$ is {\it essentially admissible locally analytic} if $V'_b$ is coadmissible as a module over $D^{\rm ess}(H, Z(G))$ for one (and hence every) open compact $H$.

We write $\Rep_{\rm top}(G)$ for the category of topological representations of $G$, with morphisms being $G$-equivariant continuous linear maps. We consider the following full subcategories:

\begin{itemize}
 \item $\Rep_{\rm cts}(G)$: continuous representations
 \item $\Rep_{\rm cts,ad}(G)$: admissible continuous representations
 \item $\Rep_{\rm top,c}(G)$: topological representations on compact type spaces
 \item $\Rep_{\rm la,c}(G)$: locally analytic representations
 \item $\Rep_{\rm la,c}^z(G)$: $Z$-tempered representations
 \item $\Rep_{\rm la,ad}(G)$: admissible locally analytic representations
 \item $\Rep_{\rm ess}(G)$: essentially admissible locally analytic representations
 \item $\Rep_{\rm cts,fd}(G)$: finite-dimensional continuous representations
 \item $\Rep_{\rm la,fd}(G)$: finite-dimensional locally analytic representations
\end{itemize}

Each of these categories is stable under passing to closed $G$-invariant submodules. The categories $\Rep_{\rm cts,ad}(G)$, $\Rep_{\rm la,ad}(G)$ and $\Rep_{\rm ess}(G)$ have the additional property that all morphisms are strict, with closed image.

The definition of $\Rep_{\rm top}$ and $\Rep_{\rm top,c}$ makes sense if $G$ is only assumed to be a semigroup. We will need one more category of representations of semigroups: if $S$ is a semigroup which contains a locally $L$-analytic subgroup $S_0$, we define $\Rep_{\rm la, c}^z(S)$ to be the full subcategory of $\Rep_{\rm top,c}(S)$ of representations which are locally analytic as representations of $S_0$, and can be written as an increasing union of $Z(S)$-invariant $BH$-subspaces. We will, in fact, only use this when either $S$ is a group (in which case the definition reduces to the definition of $\Rep_{la,c}^z$ above) or $S$ is commutative.

\begin{remark} If $V \in \Rep_{\rm top}(G)$, $V'$ naturally carries a right action of $G$. Hence we follow the conventions of \cite[\S 5.1]{emerton-memoir} by defining the algebra structures on $D(G)$ and its cousins in such a way that the Dirac distributions satisfy $\delta_g \star \delta_h = \delta_{hg}$, so all of our modules are left modules. The alternative is to consider the contragredient action on $V'$, which is the convention followed in \cite{ST-distributions, ST-admissible}; we do not adopt this approach here as we will occasionally wish to consider semigroups rather than groups.
\end{remark}

\subsection{Smooth and locally isotypical vectors}

We now present a slight generalisation of the theory of \cite[\S 7]{emerton-memoir}.

Let $G$ be a locally compact topological group and $H \trianglelefteq G$ closed. We suppose that $G$ admits a countable basis of neighbourhoods of the identity consisting of open compact subgroups; this is automatic if $G$ is locally $p$-adic analytic, for instance. Since $G$ acts on $H$ by homeomorphisms, the set of open compact subgroups of $H$ is preserved by $G$. 

\begin{definition}\label{def:stsm}
Let $V$ be an (abstract) $K$-vector space with an action of $G$. We say a vector $v \in V$ is $H$-smooth if there is an an open compact subgroup $U$ of $H$ such that $Uv = v$.
\end{definition}

Our assumptions imply that the space $V_{H-{\rm sm}}$ of $H$-smooth vectors is $G$-invariant. 

\begin{definition}[{\cite[7.1.1]{emerton-memoir}}]
Suppose $V \in \Rep_{\rm top}(G)$. We define 
\[ V_{H-{\rm st.sm}} = \varinjlim_{\substack{U \subseteq H \\ \text{$U$ open}}} V^U,\]
equipped with the locally convex inductive limit topology.
\end{definition}

Clearly $V_{H-{\rm st.sm}}$ can be identified with $V_{H-{\rm sm}}$ as an abstract $K$-vector space, but the inductive limit topology on the former is generally finer than the subspace topology on the latter. It is clear that the action of $G$ on $V$ induces a topological action on $V_{H-{\rm st.sm}}$, so $(-)_{H-{\rm st.sm}}$ is a functor from $\Rep_{\rm top}(G)$ to itself, and the natural injection $V_{H-{\rm st.sm}} \into V$ is $G$-equivariant. We say $V$ is strictly $H$-smooth if this map is a topological isomorphism.

\begin{proposition}\label{prop:stsm1}~
\begin{enumerate}[(i)]
 \item If $V \in \Rep_{\rm cts}(G)$, then $V_{H-{\rm st.sm}} \in \Rep_{\rm cts}(G)$.
 \item If $V \in \Rep_{\rm top,c}(G)$, then $V_{H-{\rm st.sm}}$ is of compact type and the natural map $V_{H-{\rm st.sm}} \to V$ is a closed embedding.
\end{enumerate}
\end{proposition}

\begin{proof}
 To show (i), we argue as in \cite[7.1.10]{emerton-memoir}. We let $G_0$ be an open compact subgroup of $G$ and $(H_i)_{i \ge 0}$ a decreasing sequence of open compact subgroups of $H$ satisfying $\bigcap_i H_i = \{1\}$ and with each $H_i$ normal in $G_0$; it is clear that we may do this, by our assumption on $G$. We set $H_i = G_i \cap H$. Then $V^{H_i}$ is a $G_0$-invariant closed subspace of $V$, and letting $V_i$ denote the kernel of the ``averaging'' map $V^{H_i} \to V^{H_{i-1}}$, we have $V^{H-{\rm st.sm}} = \bigoplus_i V_i$. Since each $V_i$ is in $\Rep_{\rm cts}(G_0)$, $V_{H-{\rm st.sm}} \in \Rep_{\rm cts}(G_0)$, which implies it is in $\Rep_{\rm cts}(G)$.
Statement (ii) depends only on $V$ as an $H$-representation, so we are reduced to the case of \cite[7.1.3]{emerton-memoir}.
\end{proof}

It follows from (ii) that for $V \in \Rep_{\rm top, c}(G)$ we do not need to distinguish between $V_{H-{\rm st.sm}}$ and $V_{H-{\rm sm}}$. Moreover, we see that if $V \in \Rep_{\rm la,c}(G)$ or any of the subcategories of admissible representations introduced above, $V_{H-{\rm st.sm}}$ has the same property.

\begin{definition}
Let $V, W$ be abstract $K$-vector spaces with an action of $G$. We say a vector $v \in V$ is locally $(H, W)$-isotypic if there is an integer $n$, an open compact subgroup $U$ of $H$, and a $U$-equivariant linear map $W^n \to V$ whose image contains $v$.
\end{definition}

The locally $(H, W)$-isotypic vectors clearly form a $G$-invariant closed subspace of $V$, since $H$ is normal in $G$. By construction, this is the image of the evaluation map $\Hom_{H-{\rm sm}}(W, V) \otimes_K W \to V$, where $\Hom_{H-{\rm sm}}(W, V)$ denotes the subspace of $H$-smooth vectors in $\Hom_K(W, V) = W' \otimes_K V$ with its diagonal $G$-action.

If $V$ and $W$ are in $\Rep_{\rm top}(G)$, with $W$ finite-dimensional, then $\Hom_K(W, V)$ has a natural topology (as a direct sum of finitely many copies of $V$) and we write $\Hom_{H-{\rm st.sm}}(W, V)$ for $\Hom_K(W, V)_{H-{\rm st.sm}}$, with its inductive limit topology as above. Then $\Hom_{H-{\rm st.sm}}(W, V) \otimes_K W$ is an object of $\Rep_{\rm top}(G)$ with a natural morphism to $V$.

We let $V_{(H, W)-{\rm liso}}$ denote the image of $\Hom_{H-{\rm st.sm}}(W, V) \otimes_K W$ in $V$, endowed with the quotient topology from the source (which is generally finer than the subspace topology on the target). We say $V$ is strictly locally $(H, W)$-isotypical if the map $V_{(H, W)-{\rm liso}} \to V$ is a topological isomorphism. 

\begin{definition}
 We say $W$ is {\bf $H$-good} if $W$ is finite-dimensional, and for any open compact subgroup $U \subseteq H$, $\End_U(W) = \End_H(W) = \End_G(W)$.
\end{definition}

\begin{proposition}\label{prop:liso-functors}
Suppose $W$ is $H$-good, with $B = \End_G(W)$. Then for any representation $V$ of $G$ on an abstract $K$-vector space, the natural map 
\[ \Hom_K(W, V)_{H-{\rm sm}} \otimes_B W \to V\]
is a $G$-equivariant injection. Dually, for any abstract right $B$-module $X$ with a $B$-linear $G$-action which is smooth restricted to $H$, the natural map
\[ X \to \Hom_K(W, X \otimes_B W)_{H-{\rm sm}}\]
is an isomorphism.
\end{proposition}

\begin{proof}
If $G = H$, the first statement is \cite[4.2.4]{emerton-memoir} (the assumption in {\it op.cit.} that $W$ be algebraic is only used to show that $W$ is $H$-good). For the general case, the map exists and is injective at the level of $H$-representations, so it suffices to note that the assumption on $W$ implies that the left-hand side has a well-defined $G$-action, for which the map is $G$-equivariant.

For the second part, it suffices to show that the map restricts to an isomorphism $X^U \to \Hom_U(W, X \otimes_B W)$ for any open $U \subseteq H$. Since $W$ is faithful as a $B$-module by construction, the natural map is an injection. Since $X$ is smooth as an $H$-representation, any vector in the left-hand side is in $\Hom_U(W, X^{U'} \otimes_B W)$ for some $U'$, which we may assume to be normal in $U$. However, we have
\[ \Hom_U(W, X^{U'} \otimes_B W) \subseteq \Hom_{U'}(W, X^{U'} \otimes_B W) = X^{U'} \otimes_B \Hom_{U'}(W, W).\]
and since $W$ is $H$-good, we have $\Hom_{U'}(W, W) = B$, so $\Hom_{U'}(W, X^{U'} \otimes_B W) = X^{U'}$. Passing to $U / U'$-invariants gives the result.
\end{proof}

Combining the preceding results shows that for $W$ an $H$-good representation, the two functors 
\[\Hom_{H-{\rm st.sm}}(W, -)\quad\text{and}\quad -\otimes_B W\]
are mutually inverse equivalences between the categories of strictly locally $(H,W)$-isotypical representations of $G$ and strictly $H$-smooth $G$-representations on right $B$-modules. 

\begin{proposition}\label{prop:stsm-embedding}
 If $H$ is a locally $L$-analytic group, and $V$ is in $\Rep_{\rm top}(G) \cap \Rep_{\rm la,c}(H)$, then there is a topological isomorphism $V_{H-{\rm st.sm}} \cong V^\mathfrak{h}$, where $\mathfrak{h}$ is the Lie algebra of $H$. More generally, if $W$ is an $H$-good locally analytic representation of $G$, $V_{(H, W)-{\rm liso}} \cong \Hom_{\mathfrak{h}}(W, V) \otimes_B W$.
\end{proposition}

\begin{proof} Clear from proposition \ref{prop:stsm1}(i), since a vector $v \in V$ is in $V_{H-{\rm sm}}$ if and only if it is $\mathfrak{h}$-invariant.
\end{proof}

\section{Preservation of admissibility}

\subsection{Spaces of invariants}

In this section we consider a group $G$ and a normal subgroup $H$, and consider the functor of $H$-invariants $V \mapsto V^H : \Rep_{\rm top}(G) \to \Rep_{\rm top}(G / H)$. Our aim is to show that this preserves the various subcategories of admissible representations introduced in the previous section.

\begin{proposition}\label{prop:invariants-continuous}
 If $V$ is an admissible Banach representation of a locally $p$-adic analytic group $G$, and $H \trianglelefteq G$ is a closed normal subgroup, then $V^H$ is an admissible Banach representation of $G/H$.
\end{proposition}

\begin{proof}
Suppose first $G$ is compact, so $D(G)$ is Noetherian. Since $H$ is normal and acts continuously on $V$, $V^H$ is a $G$-invariant closed subspace; so $(V^H)'$ is a $D(G)$-module quotient of a finitely-generated $D(G)$-module, and hence is a finitely-generated $D(G)$-module. However, the closed embedding $C(G/H) \into C(G)$ dualises to a surjection $D(G) \to D(G/H)$, and it is clear that the $D(G)$-action on $(V^H)'$ factors through this surjection. Hence $(V^H)'$ is finitely-generated over $D(G/H)$.
In the general case, let $G_0$ be a compact open subgroup of $G$ and $H_0 = G_0 \cap H$. Then $G_0 / H_0$ is an open compact subgroup of $G/H$. By the above, $V^{H_0}$ is an admissible continuous $G_0 / H_0$-representation. Since $V^H$ is a closed $G_0 / H_0$-invariant subspace of $V^{H_0}$ it is also admissible continuous as a representation of $G_0/H_0$ and hence of $G/H$.
\end{proof}

We now suppose $G$ is a locally $L$-analytic group. We write $H \trianglelefteq_L G$ to mean that $H$ is a closed normal subgroup of $G$ and the $\mathbb{Q}_p$-subspace $\operatorname{Lie}(H) \subseteq \operatorname{Lie}(G)$ is in fact an $L$-subspace, so $H$ and $G/H$ also inherit locally $L$-analytic structures.

\begin{proposition}\label{prop:invariants-admissible}
 If $V$ is an admissible locally analytic representation of $G$, and $H \trianglelefteq_L G$. Then $V^H$ is an admissible locally analytic representation of $G/H$.
\end{proposition}

\begin{proof}
As above, we may assume $G$ is compact. As in the Banach case, we note that $V^H$ is a closed $G$-invariant subspace of $V$, so it is an admissible locally analytic $G$-representation \cite[6.4(ii)]{ST-admissible} on which the action of $G$ factors through $G/H$. Hence the action of $\Dla(G)$ on $(V^H)'$ factors through $\Dla(G/H)$. Since the natural map $\Cla(G/H) \to \Cla(G)$ is a closed embedding, $\Dla(G/H)$ is a Hausdorff quotient of $\Dla(G)$ and hence a coadmissible $\Dla(G)$-module, and so by \cite[3.8]{ST-admissible} we see that $(V^H)'_b$ is coadmissible as a $D(G/H)$-module as required.
\end{proof}

We now assume that $G$ is a locally $L$-analytic group with $Z(G)$ topologically finitely generated, and $H \trianglelefteq_L G$. In this case $Z(G/H)$ may be much larger than $Z(G) / (Z(G) \cap H)$, as in the case of $\mathbb{Q}_p^\times \ltimes \mathbb{Q}_p$; so an element of $\Rep_{\rm la,c}^z(G)$ on which $H$ acts trivially need not lie in $\Rep_{\rm la,c}^z(G/H)$. Moreover, it is not obvious that $Z(G/H)$ need be topologically finitely generated if $Z(G)$ is so. We shall therefore assume that $G$ is a direct product $H \times J$, with $H, J \trianglelefteq_L G$, and $Z(H)$ and $Z(J)$ are both topologically finitely generated.

\begin{proposition}\label{prop:invariants-essadm}
In the above situation, for any essentially admissible locally analytic $G$-representation $V$, the space $V^H$ is an essentially admissible locally analytic representation of $J$.
\end{proposition}

\begin{proof} By \cite[6.4.11]{emerton-memoir}, any closed invariant subspace of an essentially admissible representation is essentially admissible; so it suffices to assume that $V = V^H$. Let $J_0 \subseteq J$ and $H_0 \subseteq H$ be open compact subgroups. Then $G_0 = J_0 \times H_0$ is an open compact subgroup of $G$. We have $Z(G) = Z(H) \times Z(J)$, and hence $\widehat{Z(G)} = \widehat{Z(H)} \times \widehat{Z(J)}$.

We now unravel the tensor products to find that the algebra 
\[D^{\rm ess}(G_0, Z(G)) = \Dla(G_0) \hotimes{\Dla(G_0 \cap Z(G))} \Can(\widehat{Z(G)})\]
decomposes as
\begin{gather*}
\left( \Dla(H_0) \hotimes{K} \Dla(J_0)\right) \hotimes{\Dla(H_0 \cap Z(H)) \hotimes{K} \Dla(J_0 \cap Z(J))} \left(\Can(\widehat{Z(H)}) \hotimes{K} \Can(\widehat{Z(J)}) \right)\\
= \left(\Dla(H_0) \hotimes{\Dla(H_0 \cap Z(H))} \Can(\widehat{Z(H)})\right) \hotimes{K} \left(\Dla(J_0) \hotimes{\Dla(J_0 \cap Z(J))} \Can(\widehat{Z(J)})\right) \\
= D^{\rm ess}(H_0, Z(H)) \hotimes{K} D^{\rm ess}(J_0, Z(J)).
\end{gather*}

By assumption, the action of $D^{\rm ess}(H_0, Z(H))$ on $V'_b$ factors through the augmentation map to $K$; so the action of $D^{\rm ess}(G_0, Z(G))$ factors through $D^{\rm ess}(J_0, Z(J))$. Since $D^{\rm ess}(J_0, Z(J))$ is a Hausdorff quotient of $D^{\rm ess}(G_0, Z(G))$, it is a coadmissible $D^{\rm ess}(G_0, Z(G))$-algebra, and thus $V'_b$ is a coadmissible $D^{\rm ess}(J_0, Z(J))$-module as required.
\end{proof}

\subsection{Admissible representations of product groups}

In this section, we'll recall the theory presented in \cite[\S 7]{emerton-memoir} of representations of groups of the form $G \times \Gamma$, where $G$ is a locally $L$-analytic group and $\Gamma$ an arbitrary locally profinite (locally compact and totally disconnected) topological group. This will allow us to give more ``global'' formulations of the results of the previous section.

Let $*$ denote one of the set $\{$``admissible Banach'', ``admissible locally analytic'', ``essentially admissible locally analytic''$\}$, so we shall speak of ``$*$-admissible representations''. Whenever we consider essentially admissible representations we will assume that the groups concerned have topologically finitely generated centre, so the concept is well-defined.

\begin{definition}[{\cite[7.2.1]{emerton-memoir}}] A $*$-admissible representation of $(G, \Gamma)$ is a locally convex $K$-vector space $V$ with an action of $G \times \Gamma$ such that
 \begin{itemize}
  \item For each open compact subgroup $U \subseteq \Gamma$, $V^U$ has property $*$ as a representation of $G$ (in the subspace topology);
  \item V is a strictly smooth $\Gamma$-representation in the sense of definition \ref{def:stsm}.
 \end{itemize}
\end{definition}

\begin{remark}
 Our terminology is slightly different from that of \cite{emerton-memoir}, where such representations are described as $*$-admissible representations of $G \times \Gamma$. We adopt the formulation above in order to avoid ambiguity when $\Gamma$ is also a locally analytic group.
\end{remark}

The results of the preceding section can be combined to prove:

\begin{proposition}
 If $G$ and $H$ are locally $L$-analytic groups, $V$ is a $*$-representation of $G \times H$, and $Z(H)$ is compact if $*$ = ``essentially admissible locally analytic'', then the space
\[ V_{H-{\rm st.sm}} = \varinjlim_{\substack{U \subseteq H\\\text{open compact}}} V^U \]
is a $*$-admissible representation of $(G, H)$.
\end{proposition}

\begin{proof}
Since the natural maps $V^{U} \into V^{U'}$ for $U' \subseteq U$ are closed embeddings, the map $V^U \into V_{H-{\rm st.sm}}$ is also a closed embedding \cite[page II.32]{bourbaki-TVS}; and its image is clearly $\left(V_{H-{\rm st.sm}}\right)^U$, so it suffices to check that $V^U$ has property $*$ for each $U$. 

 In the admissible Banach case, this is clear from proposition \ref{prop:invariants-continuous}. In the admissible locally analytic case, it likewise follows from proposition \ref{prop:invariants-admissible}. In the essentially admissible case, it suffices to note that the assumption on $Z(H)$ implies that $V$ is essentially admissible as a representation of $G \times H$ if and only if it is essentially admissible as a representation of $G \times U$ for any open compact $U \subseteq H$; so we are in the situation of proposition \ref{prop:invariants-essadm}.
\end{proof}

A slightly more general version of this applies to groups of the form $G \times H \times J$, where $G$ and $H$ are locally $L$-analytic and $J$ is an arbitrary locally compact topological group.

\begin{theorem}\label{thm:invariants-generic}
 Let $V$ be a $*$-admissible representation of $(G \times H, J)$, where $Z(H)$ is compact in the essentially admissible case. Then $V_{H-{\rm st.sm}}$ is a $*$-admissible representation of $(G, H \times J)$.
\end{theorem}

\begin{proof}
We have
\[ V_{H-{\rm st.sm}} = (V_{J-{\rm st.sm}})_{H-{\rm st.sm}} = \varinjlim_{U \subseteq H, U' \subseteq J} V^{U \times U'},\]
which is clearly a strict inductive limit; and $V^{U \times U'}$ is the $U$-invariants in the $*$-admissible $G \times H$-representation $V^{U'}$, and hence an admissible $G$-representation. The open compact subgroups of $H \times J$ of the form $U \times U'$ are cofinal in the family of all open compact subgroups, so $V_{H-{\rm st.sm}}$ is a $*$-admissible $(G, H \times J)$-representation as required. 
\end{proof}

We write $\Rep_{\rm cts, ad}(G, \Gamma)$ for the category of admissible continuous $(G, \Gamma)$-rep\-re\-sen\-ta\-tions, and similarly for the other admissibility conditions.

\subsection{Ordinary parts and Jacquet modules}

Let $\mathcal{G}$ be a connected reductive algebraic group over $L$, and $\mathcal{P}$ a parabolic subgroup of $\mathcal{G}$ with Levi factor $\mathcal{M}$. We write $\mathcal{Z} = Z(\mathcal{M})$, $\mathcal{D} = \mathcal{M}^{ss}$. We use Roman letters $G, P, M, Z, D$ for the $L$-points of these, which are locally $L$-analytic groups. Note that the multiplication map $Z \times D \to M$ has finite kernel and cokernel, and hence a representation of $M$ has property $*$ if and only if it has the corresponding property as a representation of $Z \times D$.

Suppose that $V \in \Rep_{\rm cts,adm}(G)$. We say $V$ is {\bf unitary} if the topology of $V$ can be defined by a $G$-invariant norm (or equivalently if $V$ contains a $G$-invariant separated open lattice); this is automatic if $G$ is compact, but not otherwise. The category $\Rep_{\rm u,adm}(G)$ of unitary admissible Banach representations of $G$ over $K$ is equivalent to $\operatorname{Mod}_G^{\varpi-{\rm adm}}(\mathcal{O}_K)_\mathbb{Q}$, where $\operatorname{Mod}_G^{\varpi-{\rm adm}}(\mathcal{O}_K)$ is the category considered in \cite[2.4.5]{emerton-ord1} and the subscript $\mathbb{Q}$ denotes the category with the same objects but all Hom-spaces tensored with $\mathbb{Q}$.

In \cite[\S 3]{emerton-ord1}, Emerton constructs the ordinary part functor 
\[ \Ord_P: \operatorname{Mod}_G^{\varpi-{\rm adm}}(\mathcal{O}_K) \to \operatorname{Mod}_M^{\varpi-{\rm adm}}(\mathcal{O}_K).\]
This functor is additive, so it extends to a functor 
\[\Ord_P : \Rep_{\rm u,adm}(G) \to \Rep_{\rm u,adm}(M).\]

It is easy to extend this to representations of product groups of the type considered above. Let $\Gamma$ be a locally profinite topological group, and $V$ a unitary admissible Banach $(G, \Gamma)$-representation (i.e. admitting a $G \times \Gamma$-invariant norm). We define 
\[ \Ord_P(V) = \varinjlim_{\substack{U \subseteq \Gamma \\ \text{open}}} \Ord_P(V^U).\]
Given any subgroups $U' \subseteq U$, there is an ``averaging'' map $\pi: V^{U'} \to V^U$; and we may write $V^{U'}$ as a locally convex direct sum $V^{U'} = V^U \oplus V^\pi$, where $V^\pi$ denotes the kernel of $\pi$. Since the ordinary part functor commutes with direct sums, we find that
$\Ord_P(V^{U'}) = \Ord_P(V^U) \oplus \Ord_P(V^{\pi})$; thus the natural map $\Ord_P(V^U) \to \Ord_P(V^{U'})$ is a closed embedding, and if $U' \trianglelefteq U$, we have $\Ord_P(V^{U'})^U = \Ord_P(V^U)$. Passing to the direct limit, we have $\Ord_P(V)^U = \Ord_P(V^U)$, and $\Ord_P(V)$ is an admissible Banach $(M, \Gamma)$-representation. 

An identical argument applies to the Jacquet module functor $J_P : \Rep_{\rm ess}(G) \to \Rep_{\rm ess}(M)$ of \cite{emerton-jacquet} (and indeed to any functor which preserves direct sums). Combining this with theorem \ref{thm:invariants-generic} above, we have:

\begin{proposition}\label{prop:weak-jacquet}~
\begin{enumerate}
\item If $V \in \Rep_{\rm u, ad}(G, \Gamma)$ and $W \in \Rep_{\rm cts,fd}(M)$, then 
\[\Hom_{D-{\rm st.sm}}(W, \Ord_P V) \in \Rep_{\rm cts,ad}(Z, D \times \Gamma).\]
Moreover, $\Hom_{D-{\rm st.sm}}(W, \Ord_P V)$ is unitary if $W$ is.
\item If $V \in \Rep_{\rm ess}(G, \Gamma)$ and $W \in \Rep_{\rm la, fd}(M)$, and $\mathfrak{d} = \Lie D$, then 
\[\Hom_{D-{\rm st.sm}}(W, J_P V) = \Hom_{\mathfrak{d}}(W, J_P V) \in \Rep_{\rm ess}(Z, D \times \Gamma).\]
\end{enumerate}
\end{proposition}

\section{Jacquet modules of admissible representations}

For any $V \in \Rep_{\rm ess}(G)$, proposition \ref{prop:weak-jacquet}(ii) gives us a copious supply of essentially admissible locally analytic representations of the torus $Z$: for any open compact $U \subseteq D$, and any finite-dimensional $M$-representation $W$, $\Hom_{U}(W, J_P V) = (W' \otimes_K J_P V)^U \in \Rep_{\rm ess}(Z)$. These correspond, by the equivalence of categories of \cite[2.3.2]{emerton-interpolation}, to coherent sheaves on the rigid space $\widehat Z$. For $V \in \Rep_{\rm ess}(Z)$, we will write $\Exp V$ for the support of the sheaf corresponding to $V$, a reduced rigid subspace of $\widehat Z$.

In this section, we'll prove two results describing the geometry of the rigid spaces $\Exp \Hom_U(W, J_P V)$, for $U \subseteq D$ open compact, under additional assumptions on $V$. These generalise the corresponding results in \cite{emerton-jacquet} when $P$ is a Borel subgroup.

\subsection{Compact maps}

We begin by generalising some results from \cite[\S 2.3]{emerton-jacquet} on compact endomorphisms of topological modules. Recall that a topological $K$-algebra is said to be of compact type if it can be written as an inductive limit of Banach algebras, with injective transition maps that are both algebra homomorphisms and compact as maps of topological $K$-vector spaces. If $A$ is such an algebra, then a topological $A$-module is said to be of compact type if it is of compact type as a topological $K$-vector space.

In this situation, we have the following definition of a compact morphism ({\it op.cit.}, def.~2.3.3):

\begin{definition}\label{defn:compact} A continuous $A$-linear morphism $\phi: M \to N$ between compact type topological $A$-modules is said to be {\bf $A$-compact} if there is a commutative diagram
 \begin{equation}\xymatrix{
  M \ar[rr]^\phi \ar[dr]^\alpha & & N\\
  & N_1 \ar[ur]^\beta & \\
  V \ar[ur]^\gamma \ar@/_2pc/@{-->}[uurr]
 }\end{equation}
where $N_1$ is a compact type topological $A$-module, $\alpha$ and $\beta$ are continuous $A$-linear maps, $V$ is a compact type $K$-vector space, and $\gamma$ is a continuous $K$-linear map for which $A \hotimes{K} V \to N_1$ is surjective, and the composite dashed arrow is compact as a map of compact type $K$-vector spaces.
\end{definition}

\begin{lemma}\label{lemma:compact-product}
 If $M$ is a compact type module over a compact type topological $K$-algebra $A$; $\phi: M \to M$ is an $A$-compact map; $N$ is a finitely-generated module over a finite-dimensional $K$-algebra $B$; and $\psi: N \to N$ is $K$-linear, then the map $\phi \otimes \psi: M \otimes_K N \to M \otimes_K N$ is $(A \otimes_K B)$-compact.
\end{lemma}

\begin{proof} We may assume without loss of generality that $\psi$ is the identity, by \cite[2.3.4(i)]{emerton-jacquet}. This case follows immediately by tensoring each of the spaces in the diagram with $N$.
\end{proof}

\begin{lemma}\label{lemma:compact-descent}
 Let $\sigma: A \to A'$ be a finite morphism of compact type topological $K$-algebras, and $\phi: M \to N$ a morphism of topological $A'$-modules which is $A'$-compact. Then $\phi$ is $A$-compact.
\end{lemma}

\begin{proof} By assumption, we have a diagram as in lemma \ref{defn:compact}, where the map $A' \hotimes{K} V \to N_1$ is surjective. Let $a_1, \dots, a_k$ be a set of elements generating $A'$ as an $A$-module, let $V' = V^k$, and define the map $\gamma' : V' \to N_1$ by $(v_1, \dots, v_k) \mapsto \sum a_i \gamma(v_i)$. 

Then it is clear that $1 \hotimes{} \gamma'$  gives a surjection $A \hotimes{K} V^k \to N_1$. Furthermore, the composite $\phi \circ \gamma': V' \to N$ is the map $(v_1, \dots, v_k) \mapsto \sum \beta (a_i \gamma(v))$. As $\beta$ is a morphism of $A'$-modules, this equals $\sum a_i (\beta \circ \gamma)(v)$, which is clearly compact (since $\beta \circ \gamma$ is). So the map $\gamma': V' \to N_1$ witnesses $\phi$ as an $A$-compact map.
\end{proof}

\subsection{Twisted distribution algebras}

Let $L$ be a finite extension of $\mathbb{Q}_p$, and $G$ a locally $L$-analytic group. Let $(H_n)_{n \ge 0}$ be a decreasing sequence of good $L$-analytic open subgroups of $G$, in the sense of \cite[\S 5.2]{emerton-memoir}, such that
\begin{itemize}
 \item the subgroups $H_n$ form a basis of neighbourhoods of the identity in $G$;
 \item $H_n$ is normal in $H_0$ for all $n$;
 \item the inclusion $H_{n+1} \hookrightarrow H_n$ extends to a morphism of rigid spaces between the underlying affinoid rigid analytic groups $\mathbb{H}_{n+1} \hookrightarrow \mathbb{H}_n$, which is relatively compact.
\end{itemize}

Such a sequence certainly always exists, since the choice of $H_0$ determines a Lie $\mathcal{O}_L$-lattice $\mathfrak{h}$ in the Lie algebra of $G$, and we may take $H_n$ to be the subgroup attached to the sublattice $\pi^n \mathfrak{h}$. We may use this sequence to write the topological $K$-algebra $A := \Dla(H_0) = \Cla(H_0)'_b$ as an inverse limit of the spaces $A_n := D(\mathbb{H}^\circ_n, H_0) = \left[C(H_0)_{\mathbb{H}^\circ_n-{\rm an}}\right]'_b$. For all $n$, $A_n$ is a compact type topological $K$-algebra, and the sequence $(A_n)_{n \ge 0}$ is a weak Fr\'echet-Stein structure on $A$.

We begin with a construction related to the ``untwisting isomorphism'' of \cite[3.2.4]{emerton-memoir}. Let $(\rho, W)$ be any finite-dimensional $K$-representation of $H_0$, and let $E = \End_K W$. We consider the following commutative diagram of $K$-vector spaces:
\begin{equation}\label{eq:untwist}
 \xymatrix{
& K[H_0] \otimes_K E \ar[dd]^*{\substack{g \otimes m \mapsto \\ g \otimes \rho(g) m}}_{\gamma} \\
K[H_0] \ar[ur]^{g \mapsto g \otimes 1}_{\alpha} \ar[dr]_{g \mapsto g \otimes \rho(g)}^{\beta} & \\
& K[H_0] \otimes_K E 
 }
\end{equation}

Here $\alpha$ and $\beta$ are ring homomorphisms, and although $\gamma$ is not a ring homomorphism, it satisfies the relation $\gamma(\alpha(x) y) = \beta(x) \gamma(y)$, so it intertwines the two $K[H_0]$-module structures on $K[H_0] \otimes_K E$ given by $\alpha$ and $\beta$. Furthermore $\gamma$ is clearly invertible. 

We now assume that $(\rho, W)$ is continuous (when $W$ is equipped with its unique Hausdorff locally convex topology). This implies that $W$ is locally analytic (since locally analytic vectors must be dense in $W$, by \cite[7.1]{ST-admissible}, but the topology of $W$ is the finest locally convex topology). Hence there is an integer $n(\rho)$ such that $W_{\mathbb{H}^\circ_n-{\rm an}} = W$ for all $n \ge n(\rho)$.

\begin{proposition} Let $n \ge n(\rho)$. Then there exist unique continuous maps $\alpha_n, \beta_n: A_n \to A_n \otimes_K \End(W)$ and $\gamma_n: A_n \otimes_K \End(W) \stackrel\sim\to A_n \otimes_K \End W$ extending the maps $\alpha$, $\beta$, $\gamma$ above.
\end{proposition}

\begin{proof}
Taking the (algebraic) $K$-dual of the diagram \eqref{eq:untwist}, we have a diagram 
\[\xymatrix{
& \mathcal{F}(H_0, E') \ar[dl]^{\alpha'} \\
\mathcal{F}(H_0, K)   & \\
& \mathcal{F}(H_0, E') \ar[uu]^{\gamma'} \ar[ul]_{\beta'}
 }
\]
where for $K$-vector space $V$, $\mathcal{F}(H_0, V)$ indicates the $K$-vector space of arbitrary functions $H_0 \to V$. One finds that for a function $f: H_0 \to E'$, we have $\alpha'(f)(m) = f(m)(1)$ and $\beta'(f)(m) = f(m)(\rho(m))$, while $\gamma'(f)(m) = x \mapsto f(\rho(m) x)$. All of these maps manifestly preserve the subspaces of $\mathbb{H}^\circ_n$-analytic functions for $n \ge n(\rho)$, and are continuous for the natural topologies of these subspaces; so there are corresponding maps between the duals of these subspaces, as required.
\end{proof}

\begin{corollary}
For each $n \ge n(\rho)$, the map $\beta_n$ makes $B_n = A_n \otimes_K \End W$ a finitely-generated topological $A_n$-module, and the natural map $B_{n+1} \to B_n$ induces an isomorphism $A_n \hotimes{A_{n+1}} B_{n+1} \stackrel{\sim}{\to} B_n$.
\end{corollary}

\begin{proof}
This is clearly true for the $A_n$-module structure on $B_n$ given by $\alpha_n$, so it follows for the $\beta_n$-structure (since the untwisting isomorphisms $\gamma_n$ and $\gamma_{n+1}$ are compatible with the map $B_{n+1} \to B_n$).
\end{proof}

\begin{proposition}
 Let $n \ge n(\rho)$ and let $X$ be a compact type topological $A_n$-module. Then the diagonal $H_0$-action on $X \otimes_K W$ extends to a topological $A_n$-module structure. Moreover, if $n \ge n(\rho) + 1$, we have an isomorphism of topological $A_{n-1}$-modules
\[ A_{n-1} \hotimes{A_{n}} (X \otimes_K W) \stackrel\sim\to (A_{n-1} \hotimes{A_n} X) \otimes_K W.\]
\end{proposition}

\begin{proof}
 We clearly have commuting, $K$-linear, continuous actions of $A_n$ and $\End W$ on $X \otimes_K W$, so we obtain an action of $A_n \otimes_K \End W$. Pulling back via the map $\beta_n$, we obtain an $A_n$-module structure, which clearly restricts to the diagonal action of $H_0$. The isomorphism follows from the last statement of the preceding corollary via the associativity of the tensor product, since 
\begin{align*}
&A_{n-1} \hotimes{A_{n}} (X \otimes_K W) \\
=& (A_{n-1} \hotimes{A_{n}} B_{n}) \hotimes{B_{n}} (X \otimes_K W)\\
=& B_{n-1} \hotimes{B_n} (X \otimes_K W)\\
=& (A_{n-1} \hotimes{A_n} X) \otimes_K W.\qedhere
\end{align*}
\end{proof}

\subsection{Twisted Jacquet modules}

We now return to the situation considered above, so $G$ is the group of $L$-points of a reductive algebraic group $\mathcal{G}$ over $L$ as above, with $P$ a parabolic subgroup, $M$ a Levi subgroup of $P$, $N$ the unipotent radical, and $Z = Z(M)$. We choose a sequence $(H_n)_{n \ge 0}$ of good $L$-analytic open subgroups of $G$ admitting rigid analytic Iwahori decompositions $\mathbb{H}_n = \mathbb{N}_n \times \mathbb{M}_n \times \mathbb{N}_n$, as in \cite[4.1.6]{emerton-jacquet}. We also impose the additional condition that $\mathbb{M}_n = \mathbb{Z}_n \times \mathbb{D}_n$ where $\mathbb{Z}_n$ and $\mathbb{D}_n$ are the affinoid subgroups underlying good analytic open subgroups of $Z$ and of $D = M^{ss}$; it is clear that we can always do this (by exactly the same method as in Emerton's case). We let $Z^{+}$ be the submonoid $\{z \in Z(M) : z N_0 z^{-1} \subseteq N_0\}$ of $Z$.

Our starting point is the following, which is part of the proof of \cite[4.2.23]{emerton-jacquet}:

\begin{proposition}
Let $V$ be an admissible locally analytic representation of $G$. Then for all $n \ge 0$, the action of $M_0 \times Z^+$ on the space
\[ U_n = \left( D(\mathbb{H}^\circ_n, H_0) \hotimes{\Dla(H_0)} V'_b \right)_{N_0}\]
extends to an $A_n[Z^+]$-module structure. Moreover, the transition map $A_n \hotimes{A_{n+1}} U_{n+1} \to U_n$ is $A_n$-compact and $Z^+$-equivariant, and there is some $z \in Z^{+}$ (independent of $n$) such that there exists a map $\alpha: U_n \to A_{n} \hotimes{A_{n+1}} U_{n+1}$ making the following diagram commute:
\begin{equation}\label{eq:link}
 \xymatrix{
 A_n \hotimes{A_{n+1}} U_{n+1} \ar[r] \ar[d]^{\mathrm{id} \hotimes{} z} & U_n \ar[d]^z \ar@{-->}[ld]^\alpha\\
 A_n \hotimes{A_{n+1}} U_{n+1} \ar[r] & U_n
.}
\end{equation}
\end{proposition}

We now let $\tilde U_n = U_n \otimes_K W$, where $(W, \rho)$ is a fixed, finite-dimensional, continuous representation of $M$. By the last proposition of the preceeding section, we have a diagonal $A_n$-module structure on $\tilde U_n$, and there is also a diagonal action of $Z^+$ on $\tilde U_n$ commuting with the $M_0$-action.

\begin{proposition} For any $n \ge n(\rho)$ the following holds:
\begin{itemize}
 \item $\tilde U_n$ is a compact type topological $A_n$-module, and the action of $Z^+$ is $A_n$-linear.
 \item There is an $A_{n+1}[Z^+]$-linear map $U_{n+1} \to U_n$ such that the induced map $A_n \hotimes{A_{n+1}} \tilde U_{n+1} \to \tilde U_n$ is $A_n$-compact.
 \item For any good $z \in Z^+$, we can find a map $\tilde \alpha: U_n \to A_n \hotimes{A_{n+1}} \tilde U_{n+1}$ such that the diagram corresponding to \eqref{eq:link} commutes.
\end{itemize}
Also, the direct limit $\varprojlim U_n$ (with respect to the transition maps above) is isomorphic as a topological $A[Z^+]$-module to $(V^{N_0} \otimes W')'_b$.
\end{proposition}

\begin{proof}
 Since $\tilde U_n$ is isomorphic to $(U_n)^{\oplus \dim W}$ as a topological $K$-vector space, it is certainly of compact type, and we have already observed that it is a topological $A_n$-module for all $n \ge n(\rho)$. Furthermore the $Z^+$-action commutes with the $M_0$-action, and thus it must be $A_n$-linear by continuity.

Moreover, we have an $A_n$-compact map $A_n \hotimes{A_{n+1}} U_{n+1} \to U_n$. Tensoring with the identity map gives a morphism of $A_n \otimes \End W$-modules $\left(A_n \hotimes{A_{n+1}} U_{n+1}\right) \otimes_K W \to U_n \otimes_K W$, which is $A_n \otimes_K \End W$-compact by lemma \ref{lemma:compact-product}. But the map $\beta: A_n \to A_n \otimes_K \End W$ is a finite morphism, so by lemma \ref{lemma:compact-descent}, this map is $A_n$-compact.

Finally, we know that there exists a map $\alpha: U_n \to A_n \hotimes{A_{n+1}} U_{n+1}$ through which $z$ factors, and it is clear that if we define $\tilde \alpha$ to be the map $\alpha \otimes \rho(z)$ then the diagram corresponding to \eqref{eq:link} commutes.
\end{proof}

The preceding proposition asserts precisely that the hypotheses of \cite[3.2.24]{emerton-jacquet} are satisfied, and that proposition (and its proof) give us the following:

\begin{corollary}\label{cor:coadm1}
The space $X = \left[(V^{N_0} \otimes_K W')_{\rm fs}\right]'_b$ is a coadmissible $\Can(\widehat Z) \hotimes{K} A$-module, where $(-)_{\rm fs}$ denotes the finite-slope-part functor $\Rep_{\rm top,c}(Z^+) \to \Rep_{\rm la,c}^z(Z)$ of \cite[3.2.1]{emerton-jacquet}.

Moreover, if $(Y_n)_{n \ge 0}$ is any increasing sequence of affinoid subdomains of $\widehat Z$ whose union is the entire space, then for any $n \ge n(\rho)$ we have
\[ \left(\Can(Y_n)^\dag \hotimes{K} A_n\right) \hotimes{(\Can(\widehat Z) \hotimes{K} A)} \left[(V^{N_0} \otimes_K W)_{\rm fs}\right]'_b = \Can(Y_n)^\dag \hotimes{K[Z^+]} \tilde U_n.
\]
\end{corollary}

By \cite[3.2.9]{emerton-memoir} we have $X = \left[(V^{N_0} \otimes_K W')_{\rm fs}\right]'_b = \left[(V^{N_0})_{\rm fs} \otimes_K W'\right]'_b = \left[J_P(V) \otimes_K W'\right]'_b$, so the preceding proposition gives us a description of the strong dual of the $W$-twisted Jacquet module.

\begin{theorem}
Suppose $V$ is an admissible locally analytic $G$-representation, $W$ is a finite-dimensional locally analytic representation of $M$, and $\Gamma$ is an open compact subgroup of $D$. Let $E \subseteq \widehat Z$ be the support of the coherent sheaf on $\widehat Z$ corresponding to the essentially admissible locally analytic $Z$-representation $\left(J_P(V) \otimes_K W'\right)^{\Gamma}$. Then the natural map $E \to (\operatorname{Lie} Z)'$ (induced by the differentiation map $\widehat Z \to (\operatorname{Lie} Z)'$) has discrete fibres.
\end{theorem}

\begin{proof}
Since we are free to replace the sequence $(H_n)$ of subgroups of $G$ with a cofinal subsequence, we may assume that $\Gamma \supseteq D_0$. So it suffices to prove the result for $\Gamma = D_0$. Furthermore, since the differentiation map $\widehat Z_0 \to (\operatorname{Lie} Z)'$ has discrete fibres, it suffices to show that for any character $\chi$ of $Z_0$, the rigid space 
\[ \operatorname{Exp} \left(J_P(V) \otimes_K W'\right)^{D_0, Z_0 = \chi} \subseteq \widehat{Z}\]
is discrete. If $\chi$ does not extend to a character of $M$, then this space is clearly empty, so there is nothing to prove; otherwise, let us fix such an extension, which gives us an isomorphism $\left(J_P(V) \otimes_K W'\right)^{D_0, Z_0 = \chi} = \left[J_P(V) \otimes_K (W \otimes_K \chi)'\right]^{M_0}$. So we may assume without loss of generality that $\chi$ is the trivial character, and it suffices to show that
\[ \Can(Y_n)^\dag \hotimes{\Can(\widehat Z)} \left[\left(J_P(V) \otimes_K W'\right)^{M_0}\right]'_b\]
is finite-dimensional over $K$ for all $n$, or (equivalently) all sufficiently large $n$.

If we take the completed tensor product of both sides of the formula in corollary \ref{cor:coadm1} with $\Can(Y_n)^\dag$, regarded as a $\Can(Y_n)^\dag \hotimes{K} A_n$-algebra via the augmentation map $A_n \to K$, we have
\begin{multline}\label{eq:coadm2}
 \Can(Y_n)^\dag \hotimes{\left(\Can(Y_n)^\dag \hotimes{K} A_n\right)}\left(\Can(Y_n)^\dag \hotimes{K} A_n\right) \hotimes{\left(\Can(\widehat Z) \hotimes{K} A\right)} \left[J_P(V) \otimes_K W'\right]'_b = \\ \Can(Y_n)^\dag \hotimes{\left(\Can(Y_n)^\dag \hotimes{K} A_n\right)} \left(\Can(Y_n)^\dag \hotimes{K[Z^+]} \tilde U_n\right).
\end{multline}

The left-hand side of \eqref{eq:coadm2} simplifies as
\begin{gather*}
 \Can(Y_n)^\dag \hotimes{\left(\Can(Y_n)^\dag \hotimes{K} A_n\right)}\left(\Can(Y_n)^\dag \hotimes{K} A_n\right) \hotimes{\left(\Can(\widehat Z) \hotimes{K} A\right)} \left[J_P(V) \otimes_K W'\right]'_b \\
= \Can(Y_n)^\dag \hotimes{\left(\Can(\widehat Z) \hotimes{K} A\right)}  \left[J_P(V) \otimes_K W'\right]'_b\\
= \Can(Y_n)^\dag \hotimes{\Can(\widehat Z)} \left( K \hotimes{A} \left[J_P(V) \otimes_K W'\right]'_b \right)\\
= \Can(Y_n)^\dag \hotimes{\Can(\widehat Z)} \left[ \left(J_P(V) \otimes_K W'\right)^{M_0} \right]'_b.
\end{gather*}

Meanwhile, the right-hand side of \eqref{eq:coadm2} is
\begin{gather*}
 \Can(Y_n)^\dag \hotimes{\left(\Can(Y_n)^\dag \hotimes{K} A_n\right)} \left(\Can(Y_n)^\dag \hotimes{K[Z^+]} \tilde U_n\right)\\
= \Can(Y_n)^\dag \hotimes{K[Z^+]} \left(K \hotimes{A_n} \tilde U_n\right). 
\end{gather*}

Any $z \in Z^+$ that induces an $A_n$-compact endomorphism of $\tilde U_n$ will induce a $K$-compact endomorphism of $K \hotimes{A_n} \tilde U_n$, by \cite[2.3.4(ii)]{emerton-jacquet}. Such a $z$ does exist, by hypothesis. Hence $\Can(Y_n)^\dag \hotimes{K[Z^+]} \left(K \hotimes{A_n} \tilde U_n\right)$ is finite-dimensional over $K$, by \cite[2.3.6]{emerton-jacquet}. Comparing the two sides of \eqref{eq:coadm2}, we are done.
\end{proof}

We also have a version of \cite[4.2.36]{emerton-jacquet} in this context.

\begin{theorem}
If $V$ is an admissible locally analytic representation of $G$ such that there is an isomorphism of $H$-representations $V \stackrel\sim\to \Cla(H)^r$, for some open compact $H \subseteq G$ and some $r \in \mathbb{N}$, then for any $W$ and $\Gamma$, $E = \operatorname{Exp} \left(J_P(V) \otimes_K W'\right)^{\Gamma}$ is equidimensional of dimension $d$, where $d$ is the dimension of $Z$.
\end{theorem}

\begin{proof}
As in \cite{emerton-jacquet}, we may assume (by replacing the sequence $(G_n)_{n \ge 0}$ with a cofinal subsequence if necessary) that $H = H_0$ and $\Gamma \supseteq D_0$. But then we can identify $\left(J_P(V) \otimes_K W'\right)^{\Gamma}$ with a direct summand of $\left(J_P(V) \otimes_K W'\right)^{D_0}$; this identifies $\operatorname{Exp} \left(J_P(V) \otimes_K W'\right)^{\Gamma}$ with a union of irreducible components of $\operatorname{Exp} \left(J_P(V) \otimes_K W'\right)^{D_0}$. We may therefore assume that in fact $\Gamma = D_0$. As a final reduction, letting $U_n = \left(D(\mathbb{H}^\circ_n, H_0) \hotimes{\Dla(H_0)} V'_b\right)_{N_0}$ as before, we note that the untwisting isomorphism $U_n \stackrel\sim\to D(\overline{\mathbb{N}}^\circ_n, \overline{N}_0)^r \hotimes{K} A_n$ (equation 4.2.39 in \cite{emerton-memoir}) can be extended to an isomorphism $U_n \otimes_K W \to D(\overline{\mathbb{N}}^\circ_n, \overline{N}_0)^{r \dim W} \hotimes{K} A_n$. We thus assume that $W$ is the trivial representation.

Following Emerton, we choose Banach spaces $W_n$ such that the map $D(\overline{\mathbb{N}}^\circ_{n+1}, \overline{N}_0)^r \to D(\overline{\mathbb{N}}^\circ_n, \overline{N}_0)^r$ factors through $W_n$, and (exactly as in the Borel case) for a suitable $z \in Z^+$ we have
\[ J_P(V)'_b \stackrel\sim\to \varprojlim_n K\{\{z, z^{-1}\}\} \hotimes{K[z]} (W_n \hotimes{K} A_n),\]
for some $A_n$-linear action of $z$ on $W_n \hotimes{K} A_n$ which factors through $D(\overline{\mathbb{N}}^\circ_{n+1}, \overline{N}_0)^r \hotimes{K} A_n$. Taking the completed tensor product with the map $A_n \to D(\mathbb{Z}^\circ_n, Z_0)$ given by the augmentation map of $D_0$, we have 
\[ \left[J_P(V)^{D_0}\right]'_b \stackrel\sim\to \varprojlim_n K\{\{z, z^{-1}\}\} \hotimes{K[z]} W_n \hotimes{K} D(\mathbb{Z}^\circ_n, Z_0).\]

Let us write $\widehat Z_0$ as an increasing union of affinoid subdomains $(X_n)_{n \ge 0}$, such that the natural map $\Dla(Z_0) \stackrel\sim\to \Can(\widehat Z_0) \to \Can(X_n)$ factors through $D(\mathbb{Z}^\circ_n, Z_0)$. Extending scalars from $D(\mathbb{Z}_n^\circ, Z_0)$ to $\Can(\widehat Z)$ via this map, the above formula becomes
\[
\left[J_P(V)^{D_0}\right]'_b = \varprojlim_n K\{\{z, z^{-1}\}\} \hotimes{K[z]} W_n \hotimes{K} \Can(X_n).
\]
The action of $z$ on $W_n \hotimes{K} \Can(X_n)$ is a $\Can(X_n)$-compact morphism of an orthonormalizable $\Can(X_n)$-Banach module, so the result follows by the methods of \cite{buzzard-eigen}.
\end{proof}

\section{Change of parabolic}

We now consider the problem of relating the geometric objects arising from the above construction for two distinct parabolic subgroups.

\subsection{Transitivity of Jacquet functors}

Let us recall the definition of the finite-slope-part functor, which we have already seen in the previous section. We let $Z$ be a topologically finitely generated abelian locally $L$-analytic group, and $Z^+$ an open submonoid of $Z$ which generates $Z$ as a group. Then we have the following functor $\Rep_{\rm top,c}(Z^+) \to \Rep_{\rm la,c}^z(Z)$:

\begin{definition}[{\cite[3.2.1]{emerton-jacquet}}]
For any object $V \in \Rep_{\rm top, c}(Z^+)$, we define
\[ V_{\rm fs} = \mathcal{L}_{b, Z^+}(\Can(\widehat Z), V),\]
endowed with the action of $Z$ on the first factor.
\end{definition}

\begin{lemma}
Let $Z$ be a topologically finitely generated abelian group and $Y$ a closed subgroup, and suppose $Y^+$ and $Z^+$ are submonoids of $Y$ and $Z$ satisfying the conditions above, with $Y^+ \subseteq Y \cap Z^+$. Then for all $V \in \Rep_{\rm top,c}(Z^+)$, the natural map $V_{Y-{\rm fs}} \to V$ induces an isomorphism
\[ (V_{Y-{\rm fs}})_{Z-{\rm fs}} \stackrel\sim\to V_{Z-{\rm fs}}. \]
\end{lemma}

\begin{proof} Consider the canonical $Z^+$-equivariant map $V_{Z-{\rm fs}} \to V$. We note that $V_{Z-{\rm fs}}$ is in $\Rep_{\rm la.c}^z(Z)$, and hence {\it a fortiori} in $\Rep_{\rm la.c}^z(Y)$. Hence by the universal property of \cite[3.2.4(ii)]{emerton-jacquet}, the above map factors through  $V_{Y-{\rm fs}}$. The factored map is still $Z^+$-equivariant, so by a second application of the universal property it factors through $(V_{Y-{\rm fs}})_{Z-{\rm fs}}$. This gives a continuous $Z$-equivariant map $V_{Z-{\rm fs}} \to (V_{Y-{\rm fs}})_{Z-{\rm fs}}$, which is clearly inverse to the map in the statement of the proposition.
\end{proof}

Now suppose $\mathcal{P}_1$ and $\mathcal{P}_2$ are parabolic subgroups of the reductive group $\mathcal{G}$ over $L$, with $\mathcal{P}_1 \supseteq \mathcal{P}_2$. We let $\mathcal{N}_1, \mathcal{N}_2$ be their unipotent radicals, so we have a chain of inclusions $\mathcal{G} \supseteq \mathcal{P}_1 \supseteq \mathcal{P}_2 \supseteq \mathcal{N}_2 \supseteq \mathcal{N}_1$. 

Let us choose a Levi subgroup $\mathcal{M}_2$ of $\mathcal{P}_2$, so $\mathcal{P}_2 = \mathcal{N}_2 \rtimes \mathcal{M}_2$. There is a unique Levi subgroup $\mathcal{M}_1$ of $\mathcal{P}_1$ containing $\mathcal{M}_2$; and $\mathcal{P}_{12} = \mathcal{P}_2 \cap \mathcal{M}_1$ is a parabolic subgroup of $\mathcal{M}_1$ of which $\mathcal{M}_2$ is also a Levi factor. We write $\mathcal{Z}_1$, $\mathcal{Z}_2$ for the centres of $\mathcal{M}_1$ and $\mathcal{M}_2$. 

All of the above are reductive groups over $L$, so their $L$-points are locally $L$-analytic groups; we denote these groups of points by the corresponding Roman letters.

\begin{theorem}\label{thm:transitivity}
\begin{enumerate}
 \item For any unitary admissible continuous $G$-representation $V$, there is a unique isomorphism of admissible continuous $M_2$-representations
 \[ \Ord_{P_{12}}\left(\Ord_{P_1} V\right) = \Ord_{P_2} V\]
 commuting with the canonical lifting maps from both sides into $V^{N_2}$. 
 \item For any essentially admissible locally analytic $G$-representation $V$, there is a unique isomorphism of essentially admissible locally analytic $M_2$-representations
 \[ J_{P_{12}} \left(J_{P_1} V \right) = J_{P_2} V\]
 commuting with the canonical lifting maps.
\end{enumerate}
\end{theorem}

\begin{proof}
We begin by proving the second statement. We have $N_2 = N_1 \rtimes N_{12}$, where $N_{12} = N_2 \cap M_1$ is the unipotent radical of $P_{12}$. Let $N_{2, 0}$ be an open compact subgroup of $N_2$ which has the form $N_{1, 0} \rtimes N_{12, 0}$, for open compact subgroups of the two factors; such subgroups certainly exist, since the conjugation action of $N_1$ on $N_{12}$ is continuous.

For $i = 1, 2$ we write $M_i^+$ for the submonoid of elements $m \in M_i$ for which $m N_{i,0} m^{-1} \subseteq N_{i,0}$ and $m^{-1} \overline{N}_{i,0} m \subseteq \overline{N}_{i,0}$, and $Z_i = M_i^+ \cap Z_i$. Then we have $M_2^+ \subseteq M_1^+$, and in particular $Z_1^+ \subseteq Z_2^+$.

We have
\[ J_{P_1} V  = \mathcal{L}_{b, Z_1^+}\left(\Can(\widehat{Z}_1), V^{N_{1,0}}\right)\]
endowed with the action of $M_1 = Z_1 \times_{Z_1^+} M_1^+$ determined by the actions of $Z_1$ on $\Can(\widehat{Z}_1)$ and $M_1^+$ on $V^{N_{1,0}}$. The restriction of this action to $N_{12, 0}$ is simply the action on the right factor (since $N_{12,0} \subseteq M_{1, 0} \subseteq M_1^+$) and hence
\[ (J_{P_1} V)^{N_{12,0}} = \mathcal{L}_{b, Z_1^+} \left(\Can(\widehat{Z}_1), (V^{N_{1,0}})^{N_{12,0}}\right) = \mathcal{L}_{b, Z_1^+} \left(\Can(\widehat{Z}_1), V^{N_{2,0}}\right).\]

The Hecke operator construction of \cite[\S 3.4]{emerton-jacquet} gives two actions of $M_2^+$ on $V^{N_{2,0}}$, given respectively by $m \circ v = \pi_{N_{2,0}} m v$ and $m \circ v = \pi_{N_{12, 0}} \pi_{N_{1,0}} m v$, where the operators $\pi_{N_{i,0}}$ are the averaging operators with respect to Haar measure on the subgroups $N_{i,0}$. Since $N_{2,0} = N_{12,0} \ltimes N_{1,0}$, and the Haar measure on the product is the product of the Haar measures on the factors, these two actions coincide. Applying the preceding lemma with $Z = Z_2$ and $Y = Z_1$ gives the result.

The statement for the ordinary part functor can be proved along similar lines, but it is easier to note that the functor $\Ord_P$ is right adjoint to the parabolic induction functor $\Ind_{\overline{P}}^G$ \cite[4.4.6]{emerton-ord1}, for $\overline{P}$ an opposite parabolic to $P$. Since a composition of adjunctions is an adjunction, it suffices to check instead that $\Ind_{\overline{P}_1}^G \Ind_{\overline{P}_{12}}^{M_1} U = \Ind_{\overline{P}_2}^G U$ for any $U \in \Rep_{\rm u, adm}(M_2)$. We may identify $C(G, C(M_1, U))$ with $C(G \times M_1, U)$. Evaluating at $1 \in M_1$ gives a map to $C(G, U)$, and it is easy to check that this restricts to an isomorphism between the subspaces realising the two induced representations.
\end{proof}

\subsection{Hecke algebras and the canonical lifting}
\label{ssect:hecke}

We now turn to studying the Jacquet functor in a special case; later we will combine this with the transitivity result above to deduce a general statement. As before, let $\mathcal{G}$ be a reductive algebraic group over $L$, and let $\mathcal{H} = [\mathcal{G}, \mathcal{G}]$, a semisimple group. There is a bijection between parabolics of $\mathcal{G}$ and $\mathcal{H}$, given by $\mathcal{P} \mapsto \mathcal{P}' = \mathcal{P} \cap \mathcal{H}$ and $\mathcal{P}' \mapsto \mathcal{P} = N_\mathcal{G}(\mathcal{P}')$.

We also choose an opposite parabolic $\overline{\mathcal{P}}$, determining a Levi subgroup $\mathcal{M}$ of $\mathcal{P}$, and also a Levi $\mathcal{M}'$ of $\mathcal{P}'$ in the obvious way. Write $\mathcal Z_\mathcal M$, $\mathcal Z_{\mathcal M'}$ and $\mathcal Z_\mathcal G$ for the centres of these subgroups, so $\mathcal Z_\mathcal M$ is isogenous to $\mathcal Z_{\mathcal M'} \times \mathcal Z_\mathcal G$. As before, we use Roman letters for the $L$-points of these algebraic groups.

Let $H_0$ be an open compact subgroup of $H$. We say $H_0$ is {\it decomposed} with respect to $P'$ and $\overline{P}'$ if the product of the subgroups $M_0' = H_0 \cap M'$, $N_0 = H_0 \cap N$ and $\overline{N}_0 = H_0 \cap \overline{N}$ is $H_0$, for any ordering of the factors. 

We say an element $m \in M$ is {\it positive} (for $H_0$) if $m N_0 m^{-1} \subseteq N_0$ and $m^{-1} \overline{N}_0 m \subseteq \overline{N}_0$ (see \cite[\S 3.1]{bushnell-localization}). Let $M^\oplus \subseteq M$ be the monoid of positive elements, and $Z_M^\oplus$ its intersection with $Z_M$; and let $\HH^\oplus(M_0')$ denote the subalgebra of the Hecke algebra $\HH(M'_0)$ supported on $M'^+ = M^+ \cap M'$. Note that $M^\oplus$ is contained in the monoid $M^+$ of the previous section, and clearly has finite index therein.

We say an element $z \in Z_M$ is {\it strongly positive} if the sequences $z^n N_0 z^{-n}$ and $z^{-n} \overline{N}_0 z^n$ tend monotonically to $\{1\}$; if this holds, then $z^{-1}$ and $M^\oplus$ together generate $M$. Such elements exist in abundance; any element whose pairing with the roots corresponding to $P$ has sufficiently large valuation will suffice. In particular, there exist strongly positive elements in $Z_{M'}$.

\begin{proposition}
For any essentially admissible $G$-representation $V$, we have $J_P(V) = \left(V^{N_0}\right)_{Y-{\rm fs}}$, where $Y$ is any closed subgroup of $M$ that contains a strongly positive element. In particular, $J_P(V) = J_{P'}(V)$.
\end{proposition}

\begin{proof} For any open compact $N_0 \subseteq N$, \cite[lemma 3.2.29]{emerton-jacquet} and the discussion following it shows that $V^{N_0}$ is in the category denoted therein by $\operatorname{Rep}_{\rm la, c}^z(Z_M^+)$; thus the hypotheses of \cite[prop 3.2.28]{emerton-jacquet} are satisfied for the subgroup $Y = Z_{M'}$. The conclusion of that proposition then states that $J_P(V) = (V^{N_0})_{Z_M-{\rm fs}} = (V^{N_0})_{Y-{\rm fs}}$.
\end{proof}

We now lighten the notation somewhat by writing superscript $+$ for $\oplus$, since the proposition shows that the distinction between $M^+$ and $M^\oplus$ is unimportant from the perspective of Jacquet modules.

\begin{proposition}
Let $j$ be the morphism $\HH^+(M_0') \to \HH(H_0)$ constructed in \cite[\S 3.3]{bushnell-localization}. Then the natural inclusion $V^{H_0} \into V^{M_0' N_0}$ is $\HH^+(M_0')$-equivariant, where $\HH^+(M_0')$ acts via $j$ on the first space and via its inclusion into $\HH(M_0')$ on the second.
\end{proposition}

\begin{proof}
Easy check.
\end{proof}

\begin{proposition}\label{prop:comparison1}
 For any essentially admissible locally analytic $G$-representation $V$ which is smooth as an $H$-representation, the above inclusion induces an isomorphism 
\[(V^{H_0})_{Z_{M'}-{\rm fs}} \stackrel\sim\to (V^{M_0' N_0})_{Z_{M'}-{\rm fs}} = J_{P}(V)^{M_0'}.\]
Moreover, there exists a direct sum decomposition
\[ V^{H_0} = (V^{H_0})_{Z_{M'}-{\rm fs}} \oplus (V^{H_0})_{Z_{M'}-{\rm null}}\]
where the summands are closed subspaces of $V^{H_0}$, stable under the action of $Z_G$ and $\HH(M_0')$.
\end{proposition}

\begin{proof}
Let $Q = V^{M_0' N_0} / V^{H_0}$. By the left-exactness of the finite slope part functor \cite[3.2.6(ii)]{emerton-jacquet}, there is a closed embedding
\[ (V^{M_0' N_0})_{Z_{M'}-{\rm fs}} / (V^{H_0})_{Z_{M'}-{\rm fs}} \into Q_{Z_{M'}-{\rm fs}}.\]
But since $V$ is smooth as an $H$-representation, every element $v \in V^{M_0' N_0}$ is in fact in $V^{U M_0' N_0}$ for some open $U \subseteq \overline{N}$; any such $U$ contains a $Z_{M'}^+$-conjugate of $\overline{N}_0$, so there is some $z \in Z^+$ such that $z v \in V^{\overline{N}_0 M_0'}$. Our hypothesis that $H_0$ is decomposed implies that the averaging operator $\pi_{N_0}: V^{\mathfrak{n}} \to V^{N_0}$ preserves $V^{\overline{N}_0 M_0'}$, so $z \circ v = \pi_{N_0}(zv) \in V^{H_0}$. Therefore $Q$ is $Z^+_{M'}$-torsion, and thus clearly $Q_{Z-{\rm fs}} = 0$.

For the second statement, let $z$ be any strongly positive element of $Z_{M'}$. By \cite[Theorem 1]{bushnell-localization}, there exists an integer $n$ (depending only on $P$, $H_0$ and $z$) such that for any smooth $H$-representation $V$, the action of $z$ on $V^{H_0}$ via $j$ satisfies 
\[ V^{H_0} = z^n V^{H_0} \oplus \operatorname{Ker}(z^n \mid V^{H_0}),\]
with $z$ invertible on the subspace $z^n V^{H_0}$. For representations $V$ as in the statement, the subspace $\operatorname{Ker}(z^n \mid V^{H_0})$ is clearly closed, and moreover $z^n$ gives a continuous map from the essentially admissible $Z_G$-representation $V^{H_0}$ to itself, so its image is also closed.
\end{proof}

In particular, since $V^{H_0}$ is an essentially admissible $Z_G$-representation, $J_P(V)^{M_0'}$ is essentially admissible as a $Z_G$-representation, not just a representation of the larger group $Z_G \times Z_{M'} / (Z_{M'} \cap H_0)$.

\subsection{Jacquet modules of locally isotypical representations}

We now extend the results on $H$-smooth representations above to certain locally $H$-isotypical representations. 

\begin{proposition}
If $W$ is a twist of an absolutely irreducible algebraic representation of $\mathcal{G}$, and $P = MN$ is a parabolic subgroup of $G$ with $[M, M] = D$, then $\End_{\mathfrak{d}}(W^N) = K$, so in particular the $M$-representation $W^N$ is $D$-good.
\end{proposition}

\begin{proof}
The twist is of no consequence, so suppose that $W$ is algebraic. Let us choose a maximal torus $T$ in $M$, and a field $K' \supset K$ over which $M$ is split; then there is a Borel subgroup $B \subseteq P$ defined over $K'$ with Levi factor $T$. The theory of highest weights then shows that $W$ is absolutely irreducible if and only if the highest weight space of $W$ is 1-dimensional; applying this condition to $W$ and to the $M$-representation $W^N$, we deduce that $W^N$ is absolutely irreducible as an $M$-representation. Since $M$ is isogenous to $D \times Z(M)$ and all absolutely irreducible representations of $Z(M)$ are one-dimensional, it follows that $W^N$ is in fact absolutely irreducible as a $D$-representation.
\end{proof}

\begin{proposition}\label{prop:liso-jacquet1} If $W \in \Rep_{\rm la, fd}(G)$ is $H$-good, with $B = \End_{\mathfrak{h}}(W) = \End_G W$, and furthermore $W^{\mathfrak{n}} = W^N$, then for any $X \in \Rep_{\rm la, c}^z(G)$ which is smooth as an $H$-representation and has a right action of $B$, we have
\[ J_P(X \otimes_B W) = J_P(X) \otimes_B W^N.\]
\end{proposition}

\begin{proof} Compare \cite[4.3.4]{emerton-jacquet}. Since $X$ is smooth as an $H$-representation it is certainly smooth as an $N$-representation. Arguing as in the proof of proposition \ref{prop:liso-functors}, we have $(X \otimes_B W)^{N_0} = X^{N_0} \otimes_B W^{N_0}$, which by assumption equals $X^{N_0} \otimes_B W^{N}$. Passing to finite-slope parts now yields the result.
\end{proof}

The condition $W^\mathfrak{n} = W^N$ is certainly satisfied for any $W$ that is algebraic as a representation of $N$ (since any open subgroup of $N$ is Zariski-dense). 

\begin{proposition}\label{prop:liso-jacquet2}
 Let $W$ be a twist of an absolutely irreducible algebraic representation of $G$, and let $V \in \Rep_{\rm la,c}^z(P)$ be locally $(H, W)$-isotypical. Then $J_P(V)$ is locally $(D, W^N)$-isotypical, and 
\[ \Hom_{\mathfrak{d}}(W^N,  J_P(V)) = J_P(\Hom_{\mathfrak{h}}(W, V)).\]
\end{proposition}

\begin{proof}
 Let $X = \Hom_{\mathfrak{h}}(W, V)$. By proposition \ref{prop:liso-functors}, we have $V = X \otimes_K W$; so by proposition \ref{prop:liso-jacquet1} and the remark following, $J_P(V) = J_P(X) \otimes_K W^N$. Since $W^N$ is $D$-good, we can apply the converse implication of proposition \ref{prop:liso-functors} to deduce that $J_P(X) = \Hom_{\mathfrak{d}}(W^N, J_P(V))$ as required.
\end{proof}

\subsection{Combining the constructions}

We now summarize the results of the above analysis.

\begin{theorem}\label{thm:maintheorem-local} For any $V \in \Rep_{\rm ess}(G)$, we have:
\begin{enumerate}
 \item For any parabolic subgroup $P \subseteq G$ with Levi subgroup $M$, any finite-dimensional $W \in \Rep_{\rm la,c}(M)$, and any open compact subgroup $U \subseteq D = [M, M]$, there is a coherent sheaf $\mathcal{F}(V, P, W, U)$ on $\widehat{Z(M)}$ with a right action of $\HH(U)$, whose fibre at a character $\chi \in \widehat{Z(M)}$ is isomorphic (as a right $\HH(U)$-module) to the dual of the space $\Hom_{U}(W, J_P V)^{Z(M) = \chi}$. In particular, a character $\chi$ lies in the subvariety $\mathcal{S}(V, P, W, U) = \operatorname{support} \mathcal{F}(V, P, W, U)$ if and only if this eigenspace is nonzero. 
 \item If $V \in \Rep_{\rm la,ad}(G)$, then the projection $\mathcal{S}(V, P, W, U) \to (\operatorname{Lie} Z)'$ has discrete fibres.
 \item If $V$ is isomorphic as an $H$-representation to $\Cla(H)^m$ for some $m$ and some open compact $H \subseteq G$, then $\mathcal{S}(V, P, W, U)$ is equidimensional of dimension $\dim Z$.
 \item If $P_1, P_2$ are parabolics with $P_1 \supseteq P_2$ as above, $W$ is an absolutely irreducible algebraic representation of $M_1$, and $U$ is an open compact subgroup of $D_1$ which is decomposed with respect to the parabolic $P_{2} \cap D_1$, then there is a decomposition
\[ \mathcal{F}(V, P_1, U, W) = \mathcal{F}(V, P_1, U, W)_{Z_2-{\rm null}} \oplus \mathcal{F}(V, P_1, U, W)_{Z_2-{\rm fs}},\]
where the latter factor is isomorphic to a quotient of the pushforward to $Z_1$ of the sheaf 
\[ \mathcal{F}(V, P_2, W^N, U \cap D_2)\]
on $Z_2$.
\end{enumerate}
\end{theorem}

\begin{proof}
 The only statement still requiring proof is the last one. Let $Y = (J_{P_1} V)_{D_1,W-{\rm liso}}$. The closed embedding $Y \into J_{P_1}(V)$ induces by functoriality a closed embedding $J_{P_{12}} Y \into J_{P_{12}}(J_{P_1} V)$. The right-hand side is simply $J_{P_{2}} V$, by theorem \ref{thm:transitivity}. Thus we have a closed embedding
\[ \Hom_{\mathfrak{d}_2}(W^N, J_{P_{12}} Y) \into \Hom_{\mathfrak{d}_2}(W^N, J_{P_{2}} V).\]
The left-hand side is isomorphic, by proposition \ref{prop:liso-jacquet2}, to $J_{P_{12}}\left[\Hom_{\mathfrak{d}_1}(W, Y)\right]$. We may now apply proposition \ref{prop:comparison1} to the $M_1$-representation $\Hom_{\mathfrak{d}_1}(W, Y) = \Hom_{\mathfrak{d}_1}(W, J_{P_1} V)$, to deduce that there is a direct sum decomposition
\[ \Hom_{U}(W, J_{P_1} V) = \Hom_{U}(W, J_{P_1} V)_{Z_2-{\rm fs}} \oplus \Hom_{U}(W, J_{P_1} V)_{Z_2-{\rm null}}\]
and the first direct summand is isomorphic as a $Z_2$-representation to a closed subspace of
\[\Hom_{U \cap M_2'}(W^N, J_{P_{12}} Y) \subseteq \Hom_{U \cap D_2}(W^N, J_{P_{2}} V).\]
Dualising, we obtain the stated relation between the sheaves $\mathcal{F}(\dots)$.
\end{proof}

\section{Application to completed cohomology}

\subsection{Construction of eigenvarieties}

Let us now fix a number field $F$, a connected reductive group $\mathfrak{G}$ over $F$, and a prime $\p$ of $F$ above $p$. Let $\mathcal{G} = \mathfrak{G} \times_F F_\p$, a reductive group over $F_\mathfrak{p}$, and $G = \mathcal{G}(F_\p)$. Let us choose a parabolic subgroup $\mathcal{P}$ of $\mathcal{G}$ (not necessarily arising from a parabolic subgroup of $\mathfrak{G}$), and set $P = \mathcal{P}(F_\p)$, and similarly for $M, N, D, Z$ as above. We suppose our base field $L$ is a subfield of $F_\p$, so $G, P, M, N, D, Z$ are locally $L$-analytic groups. 

We recall from \cite[2.2.16]{emerton-memoir} the construction of the completed cohomology spaces $\tilde H^i$ for each cohomological degree $i \ge 0$, which are unitary admissible Banach representations of $(G, \mathfrak{G}(\mathbb{A}_f^\p) \times \pi_0)$, where $\pi_0$ is the group of components of $\mathfrak{G}(F \otimes_\mathbb{Q} \mathbb{R})$. The following is immediate from the above:

\begin{proposition} Let $\Gamma = D \times \mathfrak{G}(\mathbb{A}_f^\p) \times \pi_0$. For any $i \ge 0$, we have:
\begin{enumerate}
\item For any $W \in \Rep_{\rm cts,fd}(M)$, the space 
\[\Hom_{D-{\rm st.sm}}(W, \Ord_P \tilde H^i)\]
is an admissible continuous $(Z, \Gamma)$-representation.
\item For any $W \in \Rep_{\rm la,fd}(M)$, the space
\[ \Hom_{\mathfrak{d}}(W, J_P \tilde H^i_{\rm la})\]
is an essentially admissible locally $L$-analytic $(Z, \Gamma)$-representation.
\end{enumerate}
\end{proposition}

Let us fix an open compact subgroup $U \subseteq \Gamma$ (this is the most natural notion of a ``tame level'' in this context). Then we can use the above result to define an eigenvariety of tame level $U$, closely following \cite[\S 2.3]{emerton-interpolation}.

Let $v$ be a (finite or infinite) prime of $S$. We set 
\[ \Gamma_v =\begin{cases}  
             \mathfrak{G}(F_v) & \text{if $v \nmid \infty$ and $v \ne \p$} \\
             D & \text{if $v = \p$}\\
             \pi_0(\mathfrak{G}(F_v)) & \text{if $v \mid \infty$}.
             \end{cases}
\]
Then $\Gamma = \prod'_{v} \Gamma_v$. Let us set $U_v = U \cap \Gamma_v$. We say $v$ is {\it unramified} (for $U$) if $v$ is finite, $v \ne \p$, and $U_v$ is a hyperspecial maximal compact subgroup of $\Gamma_v$. Let $S$ be the (clearly finite) set of ramified primes, and $\Gamma^S = \prod_{v \notin S} \Gamma_v$, $\Gamma_S = \prod_{v \in S} \Gamma_v$.

It is easy to see that $U = U_S \times U^S$, where $U^S = U \cap \Gamma^S$ and similarly $U_S = U \cap \Gamma_S$, and hence we have a tensor product decomposition of Hecke algebras
\[ \HH(\Gamma // U) = \HH(\Gamma_S // U_S) \otimes \HH(\Gamma^S // U^S) =: \HH^{\rm ram} \otimes \HH^{\rm sph}.\]
As is well known, the algebra $\HH^{\rm sph}$ is commutative (but not finitely generated over $K$), while $\HH^{\rm ram}$ is finitely generated (but not commutative in general). 

By construction, $\HH(\Gamma // U)$ acts on the essentially admissible $Z$-representation $\Hom_{U}(W, J_P \tilde H^i_{\rm la})$, and hence it also acts on the corresponding sheaf $\mathcal{F}(i, P, W, U)$ on $\widehat Z$.

\begin{definition}
Let $\mathcal{E}(i, P, W, U)$ be the relative spectrum $\operatorname{\underline{Spec}} \mathcal{A}$, where $\mathcal{A}$ is the $\mathcal{O}_{\widehat Z}$-subsheaf of $\operatorname{\underline{End}} \mathcal{F}(i, P, W, U)$ generated by the image of $\HH^{\rm sph}$.
\end{definition}

For the definition of the relative spectrum, see \cite[Thm 2.2.5]{conrad06}. By definition $\mathcal{E}(i, P, W, U)$ is a rigid space over $K$, endowed with a finite morphism $\pi: \mathcal{E}(i, P, W, U) \to \widehat Z$ and an isomorphism of sheaves of $\mathcal{O}_{\widehat Z}$-algebras $\mathcal{A} \cong \pi_* \mathcal{O}_{\mathcal{E}(i, P, W, U)}$. Consequently, $\mathcal{F}(i, P, W, U)$ lifts to a sheaf $\overline{\mathcal{F}}(i, P, W, U)$ on $\mathcal{E}(i, P, W, U)$. 

We can regard $\mathcal{E}(i, P, W, U)$ as a subvariety of $\widehat Z_K \times \Spec \HH^{\rm sph}$ (although the latter will not be a rigid space if $\mathfrak{G}$ is not the trivial group); in particular, a $K$-point of $\mathcal{E}(i, P, W, U)$ gives rise to a homomorphism $\lambda: \HH^{\rm sph} \to K$.

We record the following properties of this construction, which are precisely analogous to \cite[2.3.3]{emerton-interpolation}:

\begin{theorem}\label{thm:maintheorem}~
\begin{enumerate}
\item The natural projection $\mathcal{E}(i, P, W, U) \to \mathfrak{z}'$ has discrete fibres. In particular, the dimension of $\mathcal{E}(i, P, W, U)$ is at most equal to the dimension of $Z$.
\item The action of $\HH^{\rm ram}$ on $\mathcal{F}(i, P, W, U)$ lifts to an action on $\overline{\mathcal{F}}(i, P, W, U)$, and the fibre of $\overline{\mathcal{F}}(i, P, W, U)$ at a point $(\chi, \lambda) \in \widehat Z \times \Spec \HH^{\rm sph}$ is isomorphic as a right $\HH^{\rm ram}$-module to the dual of the $(Z = \chi, \HH^{\rm sph} = \lambda)$-eigenspace of $\Hom_U(W, J_P \tilde H^i_{\rm la})$. In particular, the point $(\chi, \lambda)$ lies in $\mathcal{E}(i, P, W, U)$ if and only if this eigenspace is non-zero.
\item If there is a compact open subgroup $G_0 \subseteq G$ such that $(\tilde H^i_{\rm la})^{U^{(\p)}}$ is isomorphic as a $G_0$-representation to a finite direct sum of copies of $\Cla(G_0)$ (where $U^{(\p)} = U \cap \mathfrak{G}(\mathbb{A}_f^\p)$), then $\mathcal{E}(i, P, W, U)$ is equidimensional, of dimension equal to the rank of $Z$.
\end{enumerate}
\end{theorem}

\begin{remark} The hypothesis in the last point above is always satisfied when $i = 0$ and $\mathfrak{G}(F \otimes \mathbb{R})$ is compact, since for any open compact subgroup $U^{(\p)} \subseteq \mathfrak{G}(\mathbb{A}_f^\p)$, the image of $G(F) \cap U^{(\p)}$ in $G$ is a discrete cocompact subgroup $\Lambda$, and the $U^{(\p)}$-invariants $\tilde H^0(U^{(\p)})$ are isomorphic as a representation of $G$ and as a $\HH(U^{(\p)})$-module to $C(\Lambda \backslash G)$. This case is considered extensively in an earlier publication of the second author \cite{DL-algebraic}.
\end{remark}

Now let us suppose $G$ is split over $K$, and fix an irreducible (and therefore absolutely irreducible) algebraic representation $W$ of $M$. We let $\Pi(P, W, U)$ denote the set of irreducible smooth $\mathfrak{G}(\mathbb{A}_f) \times \pi_0$-representations $\pi_f$ such that $J_P(\pi_f)^U \ne 0$, and $\pi_f$ appears as a subquotient of the cohomology space $H^i(\mathcal{V}_{X})$ of \cite[\S 2.2]{emerton-interpolation} for some irreducible algebraic representation $X$ of $G$ with $(X')^N \cong W \otimes \chi$ for some character $\chi$. To any such $\pi_f$, we may associate the point $(\theta \chi, \lambda) \in \widehat Z \times \Spec \HH^{\rm sph}$, where $\theta$ is the smooth character by which $Z$ acts on $J_P(\pi_p)$, and $\lambda$ the character by which $\HH^{\rm sph}$ acts on $J_P(\pi_f)^U$. Let $E(i, P, W, U)_{\rm cl}$ denote the set of points of $\widehat Z \times \Spec \HH^{\rm sph}$ obtained in this way from representations $\pi_f \in \Pi(i, P, W, U)$.

\begin{corollary}\label{corr:maincorr}
 If the map \eqref{eq:edgemap} is an isomorphism for all irreducible algebraic representations $X$ such that $(X')^N$ is a twist of $W$, then $E(i, P, W, U)_{\rm cl} \subseteq \mathcal{E}(i, P, W, U)$. In particular, the Zariski closure of $E(i, P, W, U)_{\rm cl}$ has dimension at most $\dim Z$.
\end{corollary}

\begin{proof}
Let $\pi_f \in \Pi(i, P, W, U)$. Then the locally algebraic $(G, \mathfrak{G}(\mathbb{A}_f^{(\p)}) \times \pi_0)$-represen\-tation $\pi_f \otimes X'$ appears in $H^i(\mathcal{V}_X) \otimes_K X'$. By assumption, the latter embeds as a closed subrepresentation of $\widetilde H^i_{\rm la}$. The Jacquet functor is exact restricted to locally $X'$-algebraic representations (since this is so for smooth representations). Moreover, the functor $\Hom_{\mathfrak{d}}(W, -)$ is exact restricted to locally $W$-algebraic representations, so $\Hom_{\mathfrak{d}}\left[W, J_P(\pi_f) \otimes_K (X')^N\right]$ appears as a subquotient of $\Hom_{\mathfrak{d}}\left[W, J_P(\tilde H^i_{\rm la})\right]$. Since $(X')^N = W \otimes \chi$, the former space is simply $J_P(\pi_f) \otimes_K \chi$, so the point $(\theta \chi, \lambda)$ appears in $\mathcal{E}(i, P, W, U)$ as required.
\end{proof}

\begin{remarks}
\begin{enumerate}
 \item The entire construction can also be carried out with the spaces $\tilde H^i$ replaced by the compactly supported versions $\tilde H^i_{\rm c}$ or the parabolic versions $\tilde H^i_{\rm par}$; we then obtain analogues of the above proposition for the compactly supported or parabolic cohomology of the arithmetic quotients.
 \item It suffices to check that the map \eqref{eq:edgemap} is an isomorphism for $L = \mathbb{Q}_p$. This is known to hold in many cases, e.g. in degree $i = 0$ for any $\mathfrak{G}$, and in degree 1 for ${\rm GL}_2(\mathbb{Q})$ (as shown in \cite{emerton-interpolation}) or for a semisimple and simply connected group (as shown by the first author in \cite{hill-construction}). The ``weak Emerton criterion'' of \cite[defn.~2]{hill-construction} suffices to prove corollary \ref{corr:maincorr} when $W$ is not a character; this is known in many more cases, e.g. when $i = 2$ and the congruence kernel of $\mathfrak{G}$ is finite. When $W$ is a character $\chi: M \to \mathbb{G}_m$, the weak Emerton criterion implies that the points $E(i, P, W, U)_{\rm cl}$ are contained in the union of $\mathcal{E}(i, P, W, U)$ and the single point $(\chi^{-1}, 1)$.
\end{enumerate}
\end{remarks}

\begin{theorem}\label{thm:comparison}
 Suppose $P_1 \supseteq P_2$ are two parabolics, and $U = U^{(\p)} \times U_{\p}$, where $U^{(\p)} \subseteq \mathfrak{G}(\mathbb{A}_f^{(\p)}) \times \pi_0$ and $U_{\p} \subseteq D_1$ is decomposed with respect to $P_{2} \cap D_1$. Then $\mathcal{E}(i, P_1, W, U)$ is equal to a union of two closed subvarieties
\[ \mathcal{E}(i, P_1, W, U)_{P_2-{\rm fs}} \cap \mathcal{E}(i, P_1, W, U)_{P_2-{\rm null}},\]
which are respectively equipped with sheaves of $\HH^{\rm ram}$-modules $\overline{\mathcal{F}}(i, P, W, U)_{P_2-{\rm fs}}$ and $\overline{\mathcal{F}}(i, P, W, U)_{P_2-{\rm null}}$ whose direct sum is $\overline{\mathcal{F}}(i, P, W, U)$.

The element of $\HH^{\rm ram}$ corresponding to any strictly positive element of $Z_2$ acts invertibly on $\overline{\mathcal{F}}(i, P, W, U)_{P_2-{\rm fs}}$ and nilpotently on $\overline{\mathcal{F}}(i, P, W, U)_{P_2-{\rm null}}$; and there is a subvariety of $\mathcal{E}(i, P_2, W^{N_{12}}, U \cap D_2)$ whose image in $\widehat Z_1 \times \Spec \HH^{\rm sph}$ coincides with $\mathcal{E}(i, P_1, W, U)_{P_2-{\rm fs}}$.
\end{theorem}

\begin{proof}
 By theorem \ref{thm:maintheorem-local}, we may decompose $\mathcal{F}(i, P_1, W, U)$ as a direct sum of a null part and a finite slope part; this decomposition is clearly functorial, and hence it is preserved by the action of the Hecke algebra $\HH^{\rm sph}$, so we may define the spaces $\mathcal{E}(i, P_1, W, U)_{P_2-{\rm fs}}$ and $\mathcal{E}(i, P_1, W, U)_{P_2-{\rm null}}$ to be the relative spectra of the Hecke algebra acting on the two summands. 

For the final statement, we note that there is a quotient $\mathcal{Q}$ of $\mathcal{F}(i, P_2, W^{N_{12}}, U \cap D_2)$, corresponding to the $Z_2$-subrepresentation 
\[ J_{P_{12}}\left(\Hom_{\mathfrak{d}_1}(W, J_{P_1} \widetilde H^i_{\rm la})\right)^{U \cap M_2} \subseteq \Hom_{\mathfrak{d}_2}(W^{N_{12}}, J_{P_2} \widetilde H^i_{\rm la})^{U \cap D_2}\]
 such that the pushforward of $\mathcal{Q}$ to $\widehat Z_1$ is isomorphic to $\mathcal{F}(i, P_1, W, U)_{P_2-{\rm fs}}$. This isomorphism clearly commutes with the action of $\HH^{\rm sph}$ on both sides, from which the result follows.
\end{proof}

\section{Acknowledgements}

The second author would like to thank Kevin Buzzard and Matthew Emerton for illuminating conversations; the participants in the autum 2009 ``$P$-adic Analysis'' study group at DPMMS, University of Cambridge, for their assistance in learning the necessary theory; and Sarah for her love and support.

\def\cprime{$'$}
\providecommand{\bysame}{\leavevmode\hbox to3em{\hrulefill}\thinspace}
\renewcommand{\MRhref}[2]{\href{http://www.ams.org/mathscinet-getitem?mr=#1}{#2}}


\begin{thebibliography}{Eme06b}

\bibitem[Bou87]{bourbaki-TVS}
Nicolas Bourbaki, \emph{Topological vector spaces. {C}hapters 1--5}, Elements
  of Mathematics (Berlin), Springer-Verlag, Berlin, 1987, Translated from the
  French by H. G. Eggleston and S. Madan. \MR{910295}

\bibitem[Bus01]{bushnell-localization}
Colin~J. Bushnell, \emph{Representations of reductive {$p$}-adic groups:
  localization of {H}ecke algebras and applications}, J. London Math. Soc. (2)
  \textbf{63} (2001), no.~2, 364--386. \MR{1810135}

\bibitem[Buz07]{buzzard-eigen}
Kevin Buzzard, \emph{Eigenvarieties}, {$L$}-functions and {G}alois
  representations ({D}urham, 2004), London Math. Soc. Lecture Notes, vol. 320,
  Cambridge Univ. Press, 2007, pp.~59--120. \MR{2392353}

\bibitem[CM98]{CMeigen}
Robert~F. Coleman and Barry Mazur, \emph{The eigencurve}, Galois
  representations in arithmetic algebraic geometry (Durham, 1996), London Math.
  Soc. Lecture Notes, vol. 254, Cambridge Univ. Press, 1998, pp.~1--113.
  \MR{1696469}

\bibitem[Col96]{C-CO}
Robert~F. Coleman, \emph{Classical and overconvergent modular forms}, Invent.
  Math. \textbf{124} (1996), no.~1-3, 215--241. \MR{1369416}

\bibitem[Col97]{Cpadic}
\bysame, \emph{{$p$}-adic {B}anach spaces and families of modular forms},
  Invent. Math. \textbf{127} (1997), no.~3, 417--479. \MR{1431135}

\bibitem[Con06]{conrad06}
Brian Conrad, \emph{Relative ampleness in rigid geometry}, Ann. Inst. Fourier
  (Grenoble) \textbf{56} (2006), no.~4, 1049--1126. \MR{2266885}

\bibitem[Eme04]{emerton-memoir}
Matthew Emerton, \emph{Locally analytic vectors in representations of locally
  {$p$}-adic analytic groups}, Mem. Am. Math. Soc. (to appear), 2004.

\bibitem[Eme06a]{emerton-jacquet}
\bysame, \emph{Jacquet modules of locally analytic representations of
  {$p$}-adic reductive groups. {I}. {C}onstruction and first properties}, Ann.
  Sci. {\'E}cole Norm. Sup. (4) \textbf{39} (2006), no.~5, 775--839.
  \MR{2292633}

\bibitem[Eme06b]{emerton-interpolation}
\bysame, \emph{On the interpolation of systems of eigenvalues attached to
  automorphic {H}ecke eigenforms}, Invent. Math. \textbf{164} (2006), no.~1,
  1--84. \MR{2207783}

\bibitem[Eme10]{emerton-ord1}
\bysame, \emph{Ordinary parts of admissible representations of {$p$}-adic
  reductive groups. {I}. {C}onstruction and first properties}, Ast{\'e}risque
  \textbf{331} (2010), 335--381.

\bibitem[Hil07]{hill-construction}
Richard Hill, \emph{Construction of eigenvarieties in small cohomological
  dimensions for semi-simple, simply connected groups}, Doc. Math. \textbf{12}
  (2007), 363--397. \MR{2365907}

\bibitem[Loe09]{DL-algebraic}
David Loeffler, \emph{Overconvergent algebraic automorphic forms}, in
  preparation, 2009.

\bibitem[ST02a]{ST-banach}
Peter Schneider and Jeremy Teitelbaum, \emph{Banach space representations and
  {I}wasawa theory}, Israel J. Math. \textbf{127} (2002), 359--380.
  \MR{1900706}

\bibitem[ST02b]{ST-distributions}
\bysame, \emph{Locally analytic distributions and {$p$}-adic representation
  theory, with applications to {${\rm GL}\sb 2$}}, J. Am. Math. Soc.
  \textbf{15} (2002), no.~2, 443--468. \MR{1887640}

\bibitem[ST03]{ST-admissible}
\bysame, \emph{Algebras of {$p$}-adic distributions and admissible
  representations}, Invent. Math. \textbf{153} (2003), no.~1, 145--196.
  \MR{1990669}

\end{thebibliography}
\end{document}